\documentclass[11pt,a4paper]{article}
\usepackage{tikz}
\usepackage{subfigure}
\usepackage[english]{babel}
\usepackage{graphicx}
\usepackage{extarrows}

\usepackage{indentfirst, latexsym, bm, amssymb,amsmath}
\usepackage{mathrsfs}
\usepackage{amssymb}
\usepackage{pifont}
\usepackage{tikz}
\begin{document}

	\newtheorem{theorem}{Theorem}[section]
	\newtheorem{problem}[theorem]{Problem}
	\newtheorem{corollary}[theorem]{Corollary}
	\newtheorem{observation}[theorem]{Observation}
	\newtheorem{definition}[theorem]{Definition}
	\newtheorem{conjecture}[theorem]{Conjecture}
	\newtheorem{question}[theorem]{Question}
	\newtheorem{lemma}[theorem]{Lemma}
	\newtheorem{proposition}[theorem]{Proposition}
	\newtheorem{example}[theorem]{Example}
	\newenvironment{proof}{\noindent {\bf
			Proof.}}{\hfill $\square$\par\medskip}
	\newcommand{\remark}{\medskip\par\noindent {\bf Remark.~~}}
	\newcommand{\pp}{{\it p.}}
	\newcommand{\de}{\em}

\newcommand{\1}{{\uppercase\expandafter{\romannumeral1}}}
\newcommand{\2}{{\uppercase\expandafter{\romannumeral2}}}
\newcommand{\3}{{\uppercase\expandafter{\romannumeral3}}}
\newcommand{\4}{{\uppercase\expandafter{\romannumeral4}}}

\title{A stability theorem for multi-partite graphs\thanks{This work is supported by National Natural Science
Foundation of China (No. 11871222, 11901554) and Science and Technology Commission
of Shanghai Municipality (No. 18dz2271000). E-mail addresses: 52215500039@stu.ecnu.edu.cn (W. Chen), chlu@math.ecnu.edu.cn (C. Lu), ltyuan@math.ecnu.edu.cn.}}

\author{
Wanfang Chen, Changhong Lu, Long-Tu Yuan\\
	School of Mathematical Sciences\\
	Shanghai Key Laboratory of PMMP\\
	East China Normal University\\
	Shanghai 200241, China\\
}
\date{}
\maketitle

\begin{abstract}
The Erd\H{o}s-Simonovits stability theorem is one of the most widely used theorems in extremal graph theory.
We obtain an Erd\H{o}s-Simonovits type stability theorem in multi-partite graphs.
Different from the Erd\H{o}s-Simonovits stability theorem, our stability theorem in multi-partite graphs says that if the number of edges of an $H$-free graph $G$ is close to the extremal graphs for $H$, then $G$ has a well-defined structure but may be far away to the extremal graphs for $H$.
As applications, we strengthen a theorem of  Bollob\'{a}s, Erd\H{o}s and Straus    and solve a conjecture in a stronger form posed by Han and Zhao concerning the maximum number of edges in multi-partite graphs which does not contain vertex-disjoint copies of a clique.
\end{abstract}

\section{Introduction}
\noindent Given graphs $G$ and $F$, if $G$ does not contain a copy of $F$, then we say that $G$ is $F$-free.
We call an $n$-vertex $F$-free graph with   maximum number of edges an extremal graph for $F$.
Tur\'{a}n \cite{turan1941} showed that the $n$-vertex complete $(t-1)$-partite graph with part sizes as equal as possible, denote by $T(n,t-1)$, is the unique extremal graph for $K_{t}$, where $K_{t}$ is the complete graph on $t$ vertices.
Since then, determining the extremal graphs for a given graph with some additional conditions became one of the most important topics in combinatorics.

We will consider the following extremal problem.
Let $[r]=\{1,2,\ldots,r\}$.
In the rest of this paper, we will mostly consider $n$-vertex $r$-partite graphs with a partition $\mathcal{V}=(V_1,\ldots,V_r)$ such that $|V_i|=n_i$ for $i\in[r]$ and $n_1\geq \ldots\geq n_r$, that is, we consider spanning subgraphs of $K_{n_1,\ldots,n_r}$, where $K_{n_1,\ldots,n_r}$  is the complete  $r$-partite graph containing all edges between different $V_i$'s.
Denote by ex$(n_1,\ldots,n_r,F)$ the maximum number of edges in an $F$-free $r$-partite graph with parts of sizes $n_1,\ldots,n_r$.

Researches on extremal problems in multi-partite graphs can be traced back to the 1950s.
In 1954, Zarankiewicz \cite{Zarakiewicz1954} studied ex$(n,n,K_{2,2})$.
Later, K\"{o}v\'{a}ri, S\'{o}s and Tur\'{a}n \cite{Kovari1954} gave an upper bound of ex$(n,n,K_{s,t})$.
The above two results are strongly connected to the Tur\'{a}n numbers of bipartite graphs.
Very recently, based a quantitative variant of the random algebraic method, Conlon \cite{Conlon} gave good  lower bounds for ex$(n,m,K_{s,t})$ when $n,m$ satisfy some additional conditions.
For related topics of the Zarankiewicz problems, we refer the interested readers to \cite{Conlon} and references therein.

The following problem is related to the Tur\'{a}n numbers of non-bipartite graphs.
For a set of integers $I$, let $n_I:=\sum_{i\in I} n_i$.
Given $r \geq t \geq 3$ and $k \geq 2$, let $n_1\geq \ldots \geq n_r$.
For $I\subseteq [r]$, write $m_I = \min_{i\in I} \{n_i\}$.
Given a partition $\mathcal{P}$ of $[r]$, let $n_{\mathcal{P}}= \max_{I\in \mathcal{P}}\{n_I-m_I\}.$
Define
$$f(n_1,\ldots,n_r,k,t):=\max\limits_{\mathcal{P}} \left\{(k-1)n_\mathcal{P} + \sum_{I\neq I'\in \mathcal{P}} n_I\cdot n_{I'} \right\},$$
where the maximum is taken over all partitions $\mathcal{P}$ of $[r]$ into $t-1$ parts.
In particular, define
$$f(n_1,\ldots,n_r,1,t):=\max\limits_{\mathcal{P}} \left\{  \sum_{I\neq I'\in \mathcal{P}} n_I\cdot n_{I'} \right\},$$
where the maximum is taken over all partitions $\mathcal{P}$ of $[r]$ into $t-1$ parts.

Bollob\'{a}s, Erd\H{o}s and Straus \cite{BES1974} consider the extremal graphs for $K_{t}$ in multi-partite graphs.

\begin{theorem}[Bollob\'{a}s, Erd\H{o}s and Straus \cite{BES1974}]\label{extremal number 1}
Let $r\geq t$.
Then
 $${\rm ex}(n_1,\ldots,n_r,K_{t})=f(n_1,\ldots,n_r,1,t).$$
\end{theorem}

We will characterize the extremal $(t-1)$-partitions in Theorem~\ref{extremal number 1}.
First, we introduce some notations.

Given two partitions $\mathcal{P}=(P_1,\ldots ,P_{t-1})$ and $\mathcal{V}=(V_1,\ldots,V_r)$ of $V$.
For $i\in[t-1]$  and $j\in[r]$, we say that $V_i$ is \textcolor{blue}{{\it integral}} in $P_j$ or $P_j$ is \textcolor{blue}{{\it integral}} to $V_i$ if $V_i \subseteq P_j$ and $V_i$ is \textcolor{blue}{{\it partial}} in $P_j$ or $P_j$ is  \textcolor{blue}{{\it partial}} to $V_i$ if $V_i \cap P_j\neq \emptyset$ and $V_i  \nsubseteq P_j$.

Fix $P_j$, the  \textcolor{blue}{{\it  integral part}} of $P_j$ is the union of $V_i$'s which are integral in $P_j$ and the  \textcolor{blue}{{\it partial part}}  of $P_j$ is the union of $V_i\cap P_j$'s such that $V_i$ is partial in $P_j$.
We say $P_j$ is partial to $\mathcal{V}$ (simply $P_j$ is partial) if the partial part of $P_j$ is not empty and  is  integral  to $\mathcal{V}$ (simply $P_j$ is integral) otherwise.

We say  $\mathcal{P}$ is \textcolor{blue}{{\it 1-partial}} to $\mathcal{V}$ (simply $\mathcal{P}$ is 1-partial) if each $P_j$ contains at most one partial $V_i$.
The \textcolor{blue}{\it internalization} of  $\mathcal{P}$  is a partition, denote by  $I(\mathcal{P})$, obtained from $\mathcal{P}$ by putting all vertices of each partial $V_i$ into one $P_j$ containing some vertices of $V_i$.

Furthermore, we say $\mathcal{P}$ is  \textcolor{blue}{{\it stable}} to $\mathcal{V}$ if the followings hold,
\begin{itemize}
  \item  $\mathcal{P}$ is $1$-partial to $\mathcal{V}$,
  \item  the size of the integral part of any partial class of $\mathcal{P}$ equals to each other, and is no more than the size of any integral class of $\mathcal{P}$,
  \item  the size of the integral part of any class of $\mathcal{P}$ is no less than the size of any partial class of $\mathcal{V}$, 
  \item  after removing any  integral $V_i$, the size of the resulting set is no more than the size of the  integral part of any other class of $\mathcal{P}$.
\end{itemize}


Given a graph $H$ with partitions  $\mathcal{V}$ and $\mathcal{P}$, let $H[\mathcal{P}]$ denotes the induced $(t-1)$-partite subgraph of $H$, that is, the edge set of $H[\mathcal{P}]$ is
$$\big\{ xy:xy\in E(H), x\in P_j\cap V_{i},y\in P_{j'}\cap V_{i'},i\neq i'\in[r],j\neq j'\in[t-1]   \big\}.$$


We say a partition of $\mathcal{P}$ is \textcolor{blue}{{\it an extremal $(t-1)$-partition}} for  $\mathcal{V}$ if $e(K_{n_1,\ldots,n_r}[\mathcal{P}])=f(n_1,\ldots,n_r,1,t)$.
The following theorem characterizes the extremal structures in Theorem~\ref{extremal number 1}.

\begin{theorem}\label{extremal (t-1)-partition}
Let $\mathcal{V}=(V_1,\ldots,V_r)$ such that $|V_i|=n_i$ for $i\in[r]$ and $n_1\geq \ldots\geq n_r$.
The $(t-1)$-partition $\mathcal{P}$ is an  extremal partition for $\mathcal{V}$ if and only if
\begin{itemize}
  \item  $\mathcal{P}$ is stable to $\mathcal{V}$ and
  \item  $I(\mathcal{P})$ is still an  extremal $(t-1)$-partition for $\mathcal{V}$.
\end{itemize}
\end{theorem}

Let $K_t^s=T(st,t)$.
Erd\H{o}s and Stone \cite{erdHos1946} showed that the extremal graphs for $K_t^s$ have $e(T(n,t-1))+o(n^2)$ edges.
In 1968, Erd\H{o}s and Simonovits \cite{erdHos1967,erdHos1968,Simonovits1968} proved the following well-known stability theorem.

\begin{theorem}[Erd\H{o}s-Simonovits Weak Stability Theorem \cite{erdHos1967,erdHos1968,Simonovits1968}]\label{stability theorem}
Let $F$ be a graph with chromatic number $t\geq 3$.
For every $0<\epsilon<1$, there exists a constant $\delta >0$ such that every $F$-free graph $H$ with $n\geq1/\delta$ vertices and at least $e(T(n,t-1))-\delta n^2$ edges contains a $(t-1)$-partite subgraph with at least $e(T(n,t-1))-\epsilon n^2$ edges and can be obtained from an extremal $F$-free graph by changing at most $\epsilon n^2$ edges.
\end{theorem}

Erd\H{o}s and Simonovits essentially   got a more detailed version of Theorem~\ref{stability theorem} when $H$ is an extremal graph for $F$.
\begin{theorem}[Erd\H{o}s-Simonovits Strong Stability Theorem \cite{erdHos1967,erdHos1968,Simonovits1968}]\label{s stability theorem}
Let $F$ be a graph with chromatic number $t\geq 3$.
For every $0<\epsilon<1$ there exists a constant $\delta >0$ such that every extremal $F$-free graph $H$ with $n\geq1/\delta$ vertices can be partitioned into $t-1$ classes each containing $n/(t-1)+o(n)$ vertices, and, with the exception of at most $c_\epsilon$ vertices, each vertex of $H$ is joined to at most $\epsilon n$ vertices in its own class and to all but $\epsilon n$ vertices in the other classes.
\end{theorem}

Theorems~\ref{stability theorem} and~\ref{s stability theorem} are powerful tools in extremal graph theory.
For example, applying Theorem~\ref{s stability theorem}, Erd\H{o}s and Simonovits~\cite{ES1971} determined the extremal graphs for $K_{2,2,2}$ and applying Theorem~\ref{stability theorem},  Mubayi \cite{mubayi2010}, Pikhurko and Yilma \cite{pikhurko2017} considered the supersaturation problems (the minimum number of copies of $F$  in an $n$-vertex graph $H$ on   ex$(n,F)+q$ edges) for some specific graphs.
We do not try to list more applications of Theorems~\ref{stability theorem} and~\ref{s stability theorem}.
We only mention that Theorems~\ref{stability theorem} and~\ref{s stability theorem} often help when we consider extremal problems (not only Tur\'{a}n type problems) for non-bipartite graphs.

For a set $X$ and a partition $\mathcal{P}$, we define:
$$\mathcal{P}_X:=(P_1\setminus X,\ldots, P_{t-1}\setminus X).$$

Let $A\bigtriangleup B$ stands for the symmetric difference of the sets $A$ and $B$.
Given two graphs $H$ and $F$ on the same vertex set, we say  $F$ is \textcolor{blue}{$\alpha$-$close$} to $H$ if every vertex  $v$ satisfies $|N_H(v) \bigtriangleup N_F(v)|\leq \alpha$.

For  a spanning subgraph  $G$ of $K_{n_1,\ldots,n_r}$, we say a partition $\mathcal{P}=(P_1,\ldots,P_{t-1})$ of $G$ is an \textcolor{blue}{$(X,\epsilon)$-$stable$} partition (see Figure 1) if there exists a vertex set $X$ and a small constant $0<\epsilon<1$  with $|X|\leq\epsilon n_{t-1}$  such that
\begin{itemize}
  \item   $G-X$ is $\epsilon n_{t-1}$-close to $K_{n_1,\ldots,n_r}[\mathcal{P}]-X$ and
  \item  $\mathcal{P}_X$ is stable to $\mathcal{V}_X$.
\end{itemize}

We will establish the following stability result in multi-partite graphs.

\begin{theorem}[Weak Multi-partite Stability Theorem]\label{weak stability}
Let $F$ be a graph with chromatic number $t\geq 3$.
Let $G$ be an $F$-free $r$-partite graph with parts $\mathcal{V}=(V_1,\ldots,V_r)$.
For every $0<\epsilon<1$ there exists a constant $\delta >0$ such that if $n_{t-1}\geq1/\delta$ and
$$e(G)\geq  f(n_1,\ldots,n_r,1,t)-\delta n_{t-1}^2,$$
then, after removing $\epsilon n_{t-1}^2$ edges, $G$ has an $(X,\epsilon )$-stable $(t-1)$-partition $\mathcal{P}$ and the integral part of any class of $\mathcal{P}_X$ is larger than $(1-\epsilon)n_{t-1}$.
\end{theorem}

\begin{center}
\begin{tikzpicture}[scale = 1]
\tikzstyle{every node}=[scale=1]

\draw (1,2.7) ellipse  (0.5 and 0.2);
\draw node at (1,3.2) {$X$};

\draw (-1,0) ellipse  (0.5 and 2.2);
\draw node at (-2.8,0) {$V_2$};
\draw (-1,-3.6) ellipse  (0.5 and 1.2);
\draw node at (-2.8,-3.6) {$V_1$};

\draw (3,0) ellipse  (0.5 and 2.2);
\draw node at (4.8,0) {$V_3$};
\draw (3,-3.3) ellipse (0.5 and 0.9);
\draw node at (4.8,-3.3) {$V_1$};

\draw[style=dashed] (-4,-5) rectangle (6,2.4);
\draw node at (-4.5,-1) {$\mathcal{P}_X$};

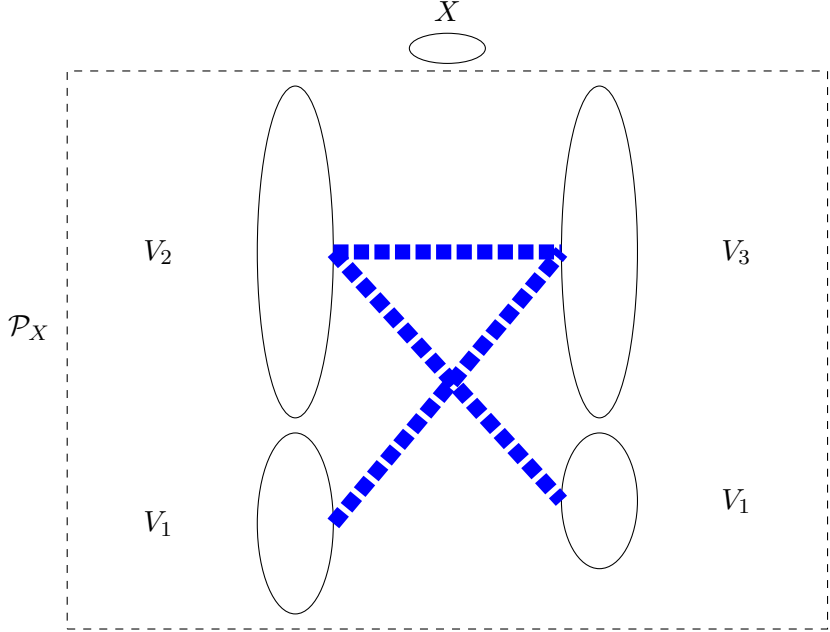
\draw node at (1,-5.5) {Figure 1. A form of $(X,\epsilon)$-stable partition $\mathcal{P}$ with $V_1$ {\it partial} in it.};

\draw [line width=0.2cm, dotted,blue] (-0.5,0)--(2.5,0);
\draw [line width=0.2cm, dotted,blue] (-0.5,-3.6)--(2.5,0);
\draw [line width=0.2cm, dotted,blue] (-0.5,0)--(2.5,-3.3);

\end{tikzpicture}
\end{center}

\remark Different from Theorem~\ref{stability theorem}, our Theorem~\ref{weak stability} shows that there exist graphs $H$ with $\chi(H)=t$ such that if $e(G)\geq  f(n_1,\ldots,n_r,1,t)-\delta n_{t-1}^2,$  then $G$ may be far away from the extremal graph for $H$, see the following example.

\medskip

\noindent{\bf Example.} Let $K=K_{n_1,n_2,n_3}$ with classes $V_1$, $V_2$ and $V_3$.
Let $n_1=m$, $n_2=m-1$ and $n_3=m-1$.
From Theorem~\ref{extremal number 1} we have ex$(n_1,n_2,n_3,K_3)=2m(m-1)$.
Moreover, from Theorems~\ref{extremal (t-1)-partition} and~\ref{strong bollobas}, we can easily deduce that the unique extremal graph is $K_{m,2(m-1)}$.
On the other hand, if a $K_3$-free subgraph $H$ of $K$ has $2m(m-1)-o(m^2)$ edges, then $H$ may be a subgraph of $K[\mathcal{P}]$, where $\mathcal{P}=(V_1 \cup X,V_2 \cup Y)$ and $(X,Y)$ is a partition of $V_3$ with $|X|=\lfloor|V_3|/2\rfloor$ and $|Y|=\lceil|V_3|/2\rceil$.
Thus $H$ is far away from (changing $m^2/2$ edges) the unique extremal graph for $K_3$.

\medskip

Given $a,b \in (0,1)$, we use $a\mathop{\ll}b$ to denote that $a<b^N$ for some sufficient large constant $N$.
Comparing with Theorem~\ref{s stability theorem}, the following theorem is useful when  we consider the extremal graphs in multi-partite graphs.

\begin{theorem}[Strong Multi-partite Stability Theorem]\label{strong stability}
Let $F$ be a graph with chromatic number $t\geq 3$.
If $G \subseteq K_{n_1,\ldots,n_r}$ is an extremal graph for $F$, then for any $\epsilon>0$ there exist   constants $\delta_1 \ll \delta_2 \ll \delta_3 \ll \epsilon$ and $c_{\epsilon,F}$ depending on $r$, $\epsilon$ and $F$ such that if $n_{t-1}\geq 1 / \delta_1$ then
\begin{itemize}
\item there exists an $(X,\delta_2)$-stable $(t-1)$-partition $\mathcal{P}$,
\item  there exists a set $Y$ with $|Y|\leq c_{\epsilon,F}$ such that every vertex $v$ of $G-Y$ satisfies $|N_G(v)\bigtriangleup N_{K_{n_1,\ldots,n_r}[\mathcal{P}]}(v)|\leq \epsilon n_{t-1}$,
 \item   the difference of the degrees of vertices of $G-Y$ in same class of $\mathcal{V}$ is at most $\epsilon n_{t-1}$ and
\item  every vertex of $Y$ is adjacent to at least  $\delta_3 n_{t-1}$ vertices of each class of $\mathcal{P}$ such that they induced a complete $(t-1)$-partite graph.
\end{itemize}

\end{theorem}

As an application of our stability theorems, we strengthen Theorem~\ref{extremal number 1} as following.

\begin{theorem}\label{strong bollobas}
Let $r\geq t$ and $n_{t-1}$ be sufficiently large.
All extremal graphs in Theorem~\ref{extremal number 1} are $(t-1)$-partite.
\end{theorem}

The second application of our stability theorems is the extremal problem of vertex-disjoint copies of a clique in multi-partite graphs.
Let $kK_{t}$ be the vertex-disjoint union of $k$ copies of $K_{t}$.
Chen, Li and Tu \cite{Chen} gave the value of  ex$(n_1,n_2,kK_2)$.
Later, De Silva, Heysse and Young \cite{Silva}  strengthened Chen, Li and Tu's result and gave a problem about ex$(n_1, \ldots,n_r,kK_t)$ when $r>t$.
In the most recently, Han and Zhao \cite{Han} determined ex$(n_1,n_2,n_3,n_4,kK_3)$ and gave the following conjecture.

\begin{conjecture}\label{-9}
Given $r \geq t \geq 3$ and $k \geq 2$, let $n_r$ be sufficiently large.
Then \begin{equation*}
{\rm ex}(n_1,\ldots,n_r,kK_{t})=f(n_1,\ldots,n_r,k,t).
\end{equation*}
\end{conjecture}
They gave a graph achieving the lower bound, which is nearly (see the remark below) the lower bond graph of a $K_{t}$-free extremal graph by adding $O(n)$ edges on it.

\remark There are extremal graphs for $kK_t$ which are not obtained by adding edges to an extremal $(t-1)$-partition.
Let $K=K_{n_1,\ldots,n_6}$ with classes $V_1,\ldots,V_6$.
Let $n_1=n_2=m+1$ and $n_3=\ldots,n_6=m$.
Note that one extremal graph for $3K_3$ is achieved by adding edges to the partition $\mathcal{P}=(V_1\cup V_2\cup V_3,V_4\cup V_5\cup V_6)$.
However $\mathcal{P}$ is not an extremal 2-partition.
\medskip

Let $\mathcal{H}$ be the family of  $(t-1)$-partite graphs in $K_{n_1,\ldots,n_r}$.
Let $\mathcal{H}^{k-1}$ be the family of graphs obtained from $H\in \mathcal{H}$ by joining all possible edges incident with $k-1$ fixed vertices.
Define
$$g(n_1,\ldots,n_r,k,t):=\max\left\{ e(F):F \in\mathcal{H}^{k-1} \right\}.$$

As an application of Theorem~\ref{strong stability}, we will confirm Han and Zhao's conjecture in the following stronger form.
\begin{theorem}\label{conjecture}
Given $r \geq t \geq 3$ and $k \geq 2$, let $n_1 \geq \ldots \geq n_r$  and $n_{t-1}$ be sufficiently large.
Then \begin{equation*}
{\rm ex}(n_1,\ldots,n_r,kK_{t})=g(n_1,\ldots,n_r,k,t).
\end{equation*}
Moreover, all extremal graphs are obtained from $(t-1)$-partite graphs by joining all possible edges incident with $k-1$ fixed vertices.
\end{theorem}

The organization of this paper is as follows.
In Section~\ref{section 2}, we introduce some lemmas.
In Section~\ref{section 3}, we present the properties of extremal graphs for $K^s_t$ and prove Theorem~\ref{extremal (t-1)-partition}.
In Sections~\ref{section 4} and \ref{section 6}, we give the proofs of Theorems~\ref{weak stability}, \ref{strong stability}, \ref{strong bollobas} and \ref{conjecture}.

\section{Preliminaries}\label{section 2}
\subsection{Definition and Notation}
\noindent Recall that we mostly consider $r$-partite graph with parts $\mathcal{V}=(V_1,\ldots,V_r)$.
Let  $V=\bigcup_{i=1}^rV_i$.
Let $x$ be any vertex of a given graph $G$, the \textcolor{blue}{{\it neighborhood}} of $x$ in $G$ is denoted by $N_G(x)=\{y\in V(G):(x,y)\in E(G)\}$.
The \textcolor{blue}{{\it degree}} of $x$ in $G$, denoted by $d_{G}(x)$, is the size of $N_G(x)$.
Given a graph $G$ and a subset $A$ of $V(G)$, let $N_G[A]$ be the set of common neighbours of $A$ in $G$.
Given two disjoint independent sets $A,B$ of $G$, let $G_{A\rightarrow B}$ be the graph obtained from $G$ by deleting edges incident with $A$ and adding each possible edge between $A$ and $B$.
For an  independent set $A$ of $G$, let  $G_{A}=G_{A\rightarrow N_G(u)}$  where $u$ is a vertex with maximum degree in $A$.

Given a partition $\mathcal{P}=(P_1,\ldots,P_{t-1})$ of $V$, we define $\mathcal{P} \wedge \mathcal{V}$ as a refinement of $\mathcal{V}$:
$$\mathcal{P} \wedge \mathcal{V}:=(P_i\cap V_j: i\in [t-1],j\in [r]).$$


We say a family of numbers $a_1,\ldots,a_s$ with $a_1\geq \ldots\geq a_s$ is \textcolor{blue}{{$L$-$balance$}}, if $a_i\geq a_{i-1}/L^4$ for each $i\in [s]$.
We partition the family of numbers $n_1,\ldots,n_r$ into maximal $L$-balance families $B_i$ such that each member of $B_{i-1}$ is larger than each member of $B_i$.
If $|B_1|\geq t-1$, then we set $ \tau(n_1,\ldots,n_r,t,L)=0$.
Otherwise, let
$$\tau(n_1,\ldots,n_r,t,L):=\max \left\{x: \sum\limits^x_{i=1}|B_i|< t-1\right\}.$$

Graph Removal Lemma is widely used in extremal graph theory and related topics.
We will apply the following simple form of the Removal Lemma.

\begin{lemma}[A simple form of the Removal Lemma \cite{Furedi2015}]\label{removal lemma}
For every $\alpha>0$, $s$ and $t\geq 3$, there is a $\delta$ such that if $n> 1/\delta$ and $G$ is an $n$-vertex $K_t^s$-free graph then it contains a $K_t$-free subgraph $H$ with $e(H)>e(G)-\alpha n^2$.
\end{lemma}

The following lemma can be found in the classic book of Bollob\'{a}s \cite{B1976}.

\begin{lemma}[Selection lemma \cite{B1976}\label{selection lemma}]
Let $\epsilon_1,\ldots,\epsilon_k$ be positive numbers, $0<\alpha<1$ and $t$ be a natural number.
There exists a natural number $N=N(\epsilon_1,\ldots,\epsilon_k;\alpha;t)$ with the following property.
For $V_1,\ldots,V_k$ with $|V_i|=n_i$ and $A_{ij}\subset V_i$ with $j\in [N]$, if
$$\frac{1}{N}\sum_{j=1}^N |A_{ij}|\geq \epsilon_i n_i \quad\mbox{ for }\quad i=1,\ldots,k,$$
then there is a subset $T\subset [N]$ with $|T|=t$, such that if $S\subseteq T$ then
$$\left|\bigcap\nolimits_{j\in S} A_{ij}\right|\geq (\alpha \epsilon_i)^{|S|}n_i \quad\mbox{ for }\quad i=1,\ldots,k.$$
\end{lemma}

The following proposition which will be used when the differences of $V_i$'s are large.
\begin{proposition}\label{size gape}
Let $N\leq \sum\nolimits_{i=1}^m u_i$.
If $ {\epsilon}<(1/m)^{10^m}$, then there exists an $\eta\in [\epsilon,\sqrt[10^m]{\epsilon}]$ such that either $u_i< \eta N$ or $u_i\geq \sqrt[10]{\eta}N$ for each $i\in[m]$.
\end{proposition}
\begin{proof}
Let $0=u_0\leq u_1\leq \ldots \leq u_m$.
If  $u_i\leq \sqrt[10^i]{\epsilon}N$ and $u_{i+1}\geq \sqrt[10^{i+1}]{\epsilon}N$ for some $i\in \{0,1,\ldots,m-1\}$, then let $\eta=\sqrt[10^i]{\epsilon}$.
Otherwise, we have $\sum\nolimits_{i=1}^m u_i\leq m\sqrt[10^m]{\epsilon}N<N$, a contradiction.
The proof of Proposition~\ref{size gape} is complete.
\end{proof}

If $\mathcal{P}$ is  stable  to $\mathcal{V}$, then $\mathcal{P} \wedge \mathcal{V} $ has some good properties.
We need the following propositions concerning the properties when $\mathcal{P}$ is  stable  to $\mathcal{V}$ or  $\mathcal{P}$ is close (see $\epsilon$-stable below) stable  to $\mathcal{V}$.

\begin{proposition}\label{larger than n_{t-1}}
Let $\mathcal{P}$ be  stable  to $\mathcal{V}$.
The size of the integral part of each class of $\mathcal{P}$ is at least  $n_{t-1}$.
\end{proposition}
\begin{proof}
If there is a class of $\mathcal{P}$ whose integral part is of size less than $n_{t-1}$, then there is an partial class of $\mathcal{V}$ whose size is at least $n_{t-1}$.
This is a contradiction to the fact that the size of the integral part of any class of $\mathcal{P}$ is larger than the size of any partial class of $\mathcal{V}$.
\end{proof}

We say that $\mathcal{P}$ is  \textcolor{blue}{{\it $\epsilon$-stable}} to $\mathcal{V}$ if the followings hold,
\begin{itemize}
  \item  $\mathcal{P}$ is $1$-partial to $\mathcal{V}$,
  \item  the sizes of integral parts  of partial classes of $\mathcal{P}$ are difference at most $\epsilon$, and are at most $\epsilon$ larger than the size of each integral class of $\mathcal{P}$,
  \item  the size of the integral part of any class of $\mathcal{P}$ is at most  $\epsilon$ smaller than the size of any partial class of $\mathcal{V}$,
  \item  after removing any integral $V_i$, the size of the resulting set is at most $\epsilon$ larger than the size of the integral part of any other class of $\mathcal{P}$.
\end{itemize}

\begin{proposition}\label{lemma for Q}
Let $\mathcal{V}=(V_1,\ldots,V_r)$ with $|V_i|=n_i$ and $n_1\geq \ldots\geq n_r$, and let $\mathcal{P}=(P_1,\ldots,P_{t-1})$ be  $\epsilon n_{t-1}$-stable to $\mathcal{V}$, where $0<\epsilon<1$ is  a small constant and $n_r$ is sufficiently large.
Suppose that each class of $\mathcal{P}\wedge \mathcal{V}$ is of size at least $\sqrt{\epsilon} n_{t-1}$.
If a vertex $v \in P_\alpha\cap V_\beta$ is adjacent to at least  $d_{K_{n_1,\ldots,n_r}[\mathcal{P}]}(v)-\epsilon n_{t-1} $ vertices of $\mathcal{V} \setminus V_\beta$, then $v$ is adjacent to at least $  (\sqrt{\epsilon}/4) n_{t-1}$ vertices of each of $t-2$ integral classes of $\mathcal{V}$ in different parts of $\mathcal{P}$.
\end{proposition}

\begin{proof}
Let $X_{i}$  be the integral part of $\mathcal{P}_{i}$ for $i\in [t-1]$ and $Y_{i}$  be the partial part of $\mathcal{P}_{i}$.
Suppose that   $v \in P_\alpha\cap V_\beta$ has less than $  (\sqrt{\epsilon}/4) n_{t-1}$ neighbours in both of $X_{i_1}$ and $X_{i_2}$.

If $V_{\beta}$ is partial in both of $P_{i_1}$ and $P_{i_2}$, then $|P_{i_1}\setminus V_{\beta}|\geq |P_{\alpha}\setminus V_{\beta}|-\epsilon n_{t-1}$ and $|P_{i_2}\setminus V_{\beta}|\geq |P_{\alpha}\setminus V_{\beta}|-\epsilon n_{t-1}$ due to $\mathcal{P}$ is $\epsilon n_{t-1}$-stable to $\mathcal{V}$.
Similar as the proof of Proposition~\ref{larger than n_{t-1}}, the size of each $X_{i}$ is at least $(1-\epsilon) n_{t-1}$.
Therefore, $v$ is adjacent to at most
$$\left|\bigcup\nolimits_{i\neq i_1,i_2} (P_i \setminus  V_\beta )\right|+ (\sqrt{\epsilon}/2)  n_{t-1}< \left|\bigcup\nolimits_{i\neq \alpha} (P_i \setminus  V_\beta )\right|- \epsilon  n_{t-1}$$
vertices of $\mathcal{V} \setminus V_\beta$, a contradiction.

If $V_{\beta}$ is integral in one of $P_{i_1},P_{i_2}$, without loss of generality, say $P_{i_1}$, then $v$ is adjacent to at most
$$\left|\bigcup\nolimits_{i\neq i_1,i_2} (P_i \setminus  V_\beta )\right|- (\sqrt{\epsilon}/2)  n_{t-1}< \left|\bigcup\nolimits_{i\neq \alpha} (P_i \setminus  V_\beta )\right|- \epsilon  n_{t-1}$$
vertices of $\mathcal{V} \setminus V_\beta$, a contradiction.

Now, without loss of generality, we may assume that $V_{\beta}$ is not partial in $P_{i_1}$ and $V_{\beta}$ is not integral in both of $P_{i_1},P_{i_2}$.
Let $V_{\beta} \cap P_{i_1}=\emptyset$, i.e., $|X_{i_1}|+ |Y_{i_1}|= |P_{i_1} \setminus V_{\beta}|$.
Since $\mathcal{P}$ is $\epsilon n_{t-1}$-stable to $\mathcal{V}$ and  each class of $\mathcal{P}\wedge \mathcal{V}$ is of size at least $\sqrt{\epsilon} n_{t-1}$, we have $|Y_{i_2}|\leq |X_{i_1} |-(\sqrt{\epsilon}-\epsilon) n_{t-1}$.
Hence $|Y_{i_1}|+|Y_{i_2}|+(\sqrt{\epsilon}/2)  n_{t-1} < |P_{i_1} \setminus  V_\beta|-(\sqrt{\epsilon}/4)  n_{t-1}$ ($V_\beta$ is not integral in $P_{i_2}$), that is $v$ is adjacent to at most
$$\left|\bigcup\nolimits_{i\neq i_1,i_2} (P_i \setminus  V_\beta )\right|+ |Y_{i_1}|+|Y_{i_2}|+(\sqrt{\epsilon}/2)  n_{t-1}< \left|\bigcup\nolimits_{i=1}\nolimits^{t-1} (P_i \setminus  V_\beta )\right|- \epsilon  n_{t-1},$$
vertices of $\mathcal{V} \setminus V_\beta$, a contradiction.
We finish the proof of Proposition \ref{lemma for Q}.
\end{proof}

\begin{lemma}\label{remove vertices to make it stable}
If the partition $\mathcal{P}=(P_1,\ldots,P_{t-1})$ is $\epsilon$-stable to $\mathcal{V}=(V_1,\ldots,V_r)$ and each class of $\mathcal{P}\wedge \mathcal{V}$ is of size at least $10tr\epsilon$ then we can remove a set $X$ with $|X|\leq 4tr\epsilon $ to ensure $\mathcal{P}_X$ is stable to $\mathcal{V}_X$.
\end{lemma}
\begin{proof}
Let $P_{i\leq m}$ be the partial class and $P_{i>m}$ be the integral class of $\mathcal{P}$.
Let $X_i$ denote the integral part of $P_i$ and $Y_i$ be the partial part of $P_i$.

Let  $m=\min\{|X_{i\leq m}|,|P_{i>m}|\}$.
Since $\mathcal{P}_X$ is $\epsilon $-stable to $\mathcal{V}$, we only need to remove at most $2\epsilon$ vertices from each $X_{i\leq m}$ to make all of them with size $m$.
Moreover, we can remove at most $4\epsilon$ vertices from each $Y_{i\leq m}$ to make sure that
\begin{itemize}
\item the size of each partial class of $\mathcal{V}$ is no more than $m$ and
\item after removing any integral $V_j$ in $P_{i\leq m}$ the size of the resulting set is no more than $m$.
\end{itemize}

Let $i>m$.
If $|P_i|\geq 4tr\epsilon+m$, then we remove $3\epsilon$ vertices from every integral $V_j$ of $P_{i}$, otherwise we do nothing.
Denote the obtained class by $\widetilde{P}_i$.
Note that each class of $\mathcal{P}\wedge \mathcal{V}$ is of size at least $10tr\epsilon$.
Clearly, after deleting all vertices of any $V_i$ in $\widetilde{P}_i$, the resulting set is of size at most $m$.
Therefore, we only need to remove $X$ with $|X|\leq 4tr\epsilon$ to ensure $\mathcal{P}_X$ be stable to $ \mathcal{V}_X$.
We complete the proof of Lemma \ref{remove vertices to make it stable}.
\end{proof}

\section{Properties of extremal graphs in multi-partite graphs}\label{section 3}

First we present some properties of the extremal graphs for $K_t$ in multi-partite graphs.

\begin{proposition}\label{pro:f(n,t)}
Let $n_1\geq\ldots \geq n_r$.
Then
$f(n_1,\ldots,n_r,1,t)\geq  f(n_1,\ldots,n_r,1,t-1)+n_{t-1}^2$.
\end{proposition}
\begin{proof}
Let $\mathcal{P}=(P_1,\ldots,P_{t-2})$ be a $(t-2)$-partition of $[r]$ which attains $f(n_1,\ldots,n_r,1,t-1)$.
Thus, by the Pigeonhole Principle, one part of $\mathcal{P}$, say $P_s$, contains at least two integers $i,j\in [t-1]$.
Hence $\mathcal{P}^\ast=(P_1,\ldots,\{i\},P_s\setminus\{i\},\ldots,P_{t-2})$ is a $(t-1)$-partition of $[r]$.
Note that $n_i\geq n_{t-1}$ and $n_j\geq n_{t-1}$.
Therefore, we have $f(n_1,\ldots,n_r,1,t-1)+n_{t-1}^2\leq f(n_1,\ldots,n_r,1,t-1)+n_i n_j \leq f(n_1,\ldots,n_r,1,t)$.
The proof of this proposition is complete.
\end{proof}

Now we present some properties of the extremal graphs for $K^s_t$.

\begin{proposition}\label{neighbour copy keeps free}
Given a  graph $F$ with $\chi(F) \geq 3$.
Let $G$ be an $r$-partite graph with parts $\mathcal{V}=(V_1,\ldots,V_r)$.
Suppose that $G$ is $F$-free with two vertex-disjoint independent sets $A,B$ with $|B| \geq |F|$.
Let $X=N_G[B]\setminus A$.
Then $G_{A \rightarrow X}$ is also $F$-free.
In particular, if $G$ is $K_t$-free, then $G_{V_i}$ is still $K_t$-free.
\end{proposition}
\begin{proof}
Suppose for a contradiction that  $G_{A \rightarrow X}$ contains a copy of $F$.
Clearly, $V(F)\cap A$ is not empty, as otherwise $G$ contains a copy of $F$, a contradiction.
We can see that $G[X]$ contains a copy of $F-A$, implying $G[X\cup B]$ contains a copy of $F$ (note that $A$ is an independent set in $G_{A \rightarrow X}$), a contradiction ($G$ is $F$-free).
The proof of this proposition is complete.
\end{proof}	

Now we are ready to prove Theorem~\ref{extremal (t-1)-partition}.

\medskip

\noindent {\bf Proof of Theorem~\ref{extremal (t-1)-partition}.}
Let  $G=K_{n_1,\ldots,n_r}[\mathcal{P}]$ and $\mathcal{P}=(P_1,\ldots,P_{t-1})$ be an extremal $(t-1)$-partition of $V(G)$ (See Figure 2).
We denote $X_{i,j}=V_i\cap P_j$ for $i\in [r],j\in [k]$.
We only prove the ``only if'' part of Theorem~\ref{extremal (t-1)-partition} since the ``if'' part is trivial.

First, if $V_i$ is non-empty in $P_j$, then in any $P_{j^\prime}$ with $j\neq j^\prime$ we have $|P_j\setminus V_i|\leq |P_{j^\prime}\setminus V_i|.$
Otherwise, $G_{X_{i,j}\rightarrow N_G[X_{i,j^\prime}]}$ has more edges than $G$, a contradiction to the maximality of $G$.
Thus if $V_i$ is partial in $P_j,P_{j^\prime}$ with $j\neq j^\prime$, then $|P_j\setminus V_i|=|P_{j^\prime}\setminus V_i|,$ and if $V_i$ is integral in $P_j$ then for every distinct $P_{j^\prime}$, we have
\begin{equation}\label{eq 0}
|P_j\setminus V_i|\leq |P_{j^\prime}|.
\end{equation}
	
	Assume there exists a part $P_{j_1}$ such that two sets $V_{i_1},V_{i_2}$ with $i_1\neq i_2$ are both partial in it.
	Let $V_{i_1}$ be partial in $P_{j_2}$ with $j_2\neq j_1$ and  $V_{i_2}$ be partial in $P_{j_3}$ with $j_3\neq j_1$ (it is possible that $j_2=j_3$).
	From the last paragraph, we have $|P_{j_1}\setminus V_{i_1}|=|P_{j_2}\setminus V_{i_1}|$ and hence $G^1=G_{X_{i_1,j_2}\rightarrow N_G[X_{i_1,j_1}]}$ has same number of edges with $G$.
    Note that $|P_{j_1}\setminus V_{i_2}|=|P_{j_3}\setminus V_{i_2}|$.
    Thus $G^1_{X_{i_2,j_1}\rightarrow N_{G^1}[X_{i_2,j_3}]}$ has more edges than $G^1$, and hence has more edges than $G$, a contradiction.
	Therefore, $\mathcal{P}$ is $1$-partial.
	
	Now, we may suppose  that $V_{i_1}$ is partial in $P_{j_1},P_{j_2}$ and $V_{i_2}$ is partial in $P_{j_3},P_{j_4}$ with $i_1\neq i_2$ and $j_1<j_2<j_3<j_4$.
    We have already proved that  $|P_{j_1}\setminus V_{i_1}|=|P_{j_2}\setminus V_{i_1}|$ and  $|P_{j_3}\setminus V_{i_2}|=|P_{j_4}\setminus V_{i_2}|$.
    It is enough to show that $|P_{j_1}\setminus V_{i_1}|=|P_{j_3}\setminus V_{i_2}|$.
    Let $G^2=G_{X_{i_1,j_2}\rightarrow N_G[X_{i_1,j_1}]}$.
    It is easy to see that $G^2$ has same number of edges with $G$.
    Let $G^3=G^2_{X_{i_2,j_3}\rightarrow N_{G^2}[P_{j_2}\setminus V_{i_1}]}$.
     Note that $G^2$ and $G^3$ are $K_t$-free.
    The maximality of $G$ implies that $G^3$ has no more number of edges than $G$, hence we have $|P_{j_3}\setminus V_{i_2}|\leq|P_{j_1}\setminus V_{i_1}|$.
    By symmetry, we also have $|P_{j_3}\setminus V_{i_2}|\geq|P_{j_1}\setminus V_{i_1}|$.
    Thus $|P_{j_3}\setminus V_{i_2}|=|P_{j_1}\setminus V_{i_1}|=m$.
    Therefore, from (\ref{eq 0}), we have $m \leq |P_i|$ for each integral $P_i$.

    Let $V_i$ be partial in $P_j$ and $P_{j^\prime}$.
    Thus $P_j\setminus V_i$ and $P_{j^\prime}\setminus V_i$ are integral part with  $|P_j\setminus V_i|=|P_{j^\prime}\setminus V_i|=m$.
    In graph $G$, we adjacent $V_i$ to all other vertices and delete all the edges between $P_j\setminus V_i$ and $P_{j^\prime}\setminus V_i$ and denote the new graph by $H$.
    Note that $H$ is still a $(t-1)$-partite graph.
    Thus $e(H)=e(G)- |P_j\setminus V_i|^2  +  |V_i| |P_{j}\setminus V_i|\leq e(G)$, implying $|V_i|\leq |P_j\setminus V_i|=m$.

    Let $V_i$ be integral in $P_\ell$ and be $V_j$ partial in $P_s$ and $P_{s^\prime}$.
    Let $G^1=G_{X_{j,s}\to N_G[X_{j,s^\prime}]}$.
    Since $|P_{s^\prime}\setminus V_j|=|P_s\setminus V_j|$, $G^1$ is still the extremal $K_t$-free graph.
     By \eqref{eq 0} we have $|P_\ell\setminus V_i|\leq |P_s\setminus V_j|=m$.
     If there is no partial class in $\mathcal{P}$, then $|P_\ell\setminus V_i| \leq  m$, as otherwise there is a $(t-1)$-partition with more edges.
    Above all, we confirmed that $\mathcal{P}$ is stable to $\mathcal{V}$.

   Since the integral part of partial class equals to each other, it is easy to check that $I(\mathcal{P})$ is still an extremal partition.
   The proof of Theorem~\ref{extremal (t-1)-partition} is complete.\hfill$\square$ \medskip

\begin{center}
\begin{tikzpicture}

\draw (0,0) ellipse  (0.5 and 2);
\draw node at (0,2.4) {$V_3$};
\draw[red] (0,-2.5) ellipse  (0.5 and 0.4) node {$V_{1}$};
\draw node at (0,-4) {$P_1$};

\draw (1.5,0) ellipse  (0.5 and 2);
\draw node at (1.5,2.4) {$V_4$};
\draw[red] (1.5,-2.9) ellipse (0.5 and 0.6) node {$V_{1}$};
\draw node at (1.5,-4) {$P_2$};

\draw (3,0) ellipse  (0.5 and 2);
\draw node at (3,2.4) {$V_5$};
\draw[red] (3,-2.7) ellipse  (0.5 and 0.5) node {$V_{1}$};
\draw node at (3,-4) {$P_3$};

\draw (4.5,0) ellipse  (0.5 and 2);
\draw node at (4.5,2.4) {$V_6$};
\draw[blue] (4.5,-2.5) ellipse  (0.5 and 0.3) node {$V_{2}$};
\draw node at (4.5,-4) {$P_4$};

\draw (6,0) ellipse  (0.5 and 2);
\draw node at (6,2.4) {$V_7$};
\draw[blue] (6,-2.8) ellipse  (0.5 and 0.7) node {$V_{2}$};
\draw node at (6,-4) {$P_5$};

\draw (7.5,0) ellipse  (0.5 and 2.5);
\draw node at (7.5,2.7) {$V_8$};
\draw node at (7.5,-4) {$P_6$};

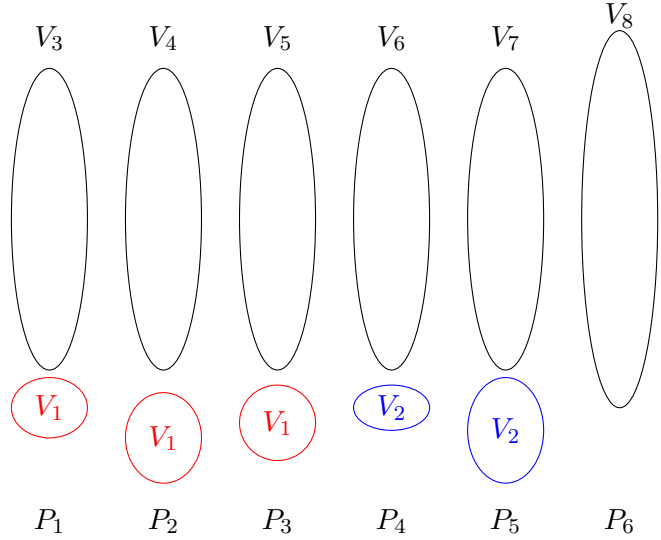
\draw node at (3.5,-5) {Figure 2. an extremal partition with $|V_1|,\ldots,|V_8|=1.5m,m,2m,\ldots,2m,2.5m$};

\end{tikzpicture}
\end{center}

\begin{proposition}\label{remove big set}
	Let $K=K_{n_1,\ldots,n_r}$ with parts $V_1,\ldots, V_{r}$ of sizes $n_1\geq \ldots\geq n_r$.
Given $L\geq r$ and let  $n_{t-1}$ be sufficiently large.
	Let $\tau=\tau(n_1,\ldots,n_r,t,L)$. If $\tau \neq 0$, then for each $0\leq s\leq \tau$, we have
$$\emph{ex}(n_1,\ldots,n_r,K_t)=\emph{ex}(n_{s+1},\ldots,n_r,K_{t-s})+\prod_{1\leq i<j\leq s}n_in_j+\sum_{i=1}^s n_i\sum_{i=s+1}^rn_i.$$
Moreover, each extremal $(t-1)$-partition of $K$ contains $V_s$ for each $1\leq s\leq \tau$.
\end{proposition}
\begin{proof}
Take an extremal $(t-1)$-partition $\mathcal{P}=(P_1,\ldots, P_{t-1})$ of $K$ with $|P_1| \geq \ldots \geq |P_{t-1}|$.
It is sufficiently to show that $P_i=V_i$ for $i=1,\ldots,\tau$.
Theorem~\ref{extremal number 1} implies that $H=K_{n_1,\ldots,n_r}[\mathcal{P}]$ is an extremal graph of $K_t$.
Since $0\leq s\leq \tau<t-1$, we have $|P_{t-1}|\leq \sum_{i=\tau +1}^{r} n_i \leq   n_{\tau}/r.$
Hence $V_s$ is integral to $\mathcal{P}$, as otherwise there is a partial class of  $\mathcal{V}$ with size larger than $|P_{t-1}|$, a contradiction to Theorem~\ref{extremal (t-1)-partition}.
If $P_j$ contains $V_s$ and some other vertices, then the size of the integral part of $P_j$ is larger than $|P_{t-1}|$, which is also a contradiction to Theorem~\ref{extremal (t-1)-partition}.
Therefore, we must have $P_i=V_i$ for $i=1,\ldots,\tau$.
The proof of Proposition \ref{remove big set} is complete.
\end{proof}

\section{Proof of Theorem \ref{weak stability} and \ref{strong stability}}\label{section 4}

\noindent In order to prove Theorem~\ref{weak stability}, we first prove the following stability result for $K_t$-free graphs.

\begin{theorem}\label{weak stability n case}
Let $G$ be a $K_t$-free $r$-partite graph with parts $\mathcal{V}=(V_1,\ldots,V_r)$ of size $n_1\geq \ldots\geq n_r$.
For any $0<\epsilon<1 $, there exists $\delta>0$ depending on $t$, $r$ and $\epsilon$ such that if $n_{t-1}\geq 1/\delta$ and $e(G)\geq f(n_1,\ldots,n_r,1,t)-\delta n_{t-1}^2$, then $G$ has an $(X,\epsilon)$-stable $(t-1)$-partition  $\mathcal{P}$.
Moreover, the size of the integral part of each class of $(\mathcal{P} \wedge \mathcal{V})_X$ is larger than $(1-\epsilon)n_{t-1}$.
\end{theorem}

\begin{proof}
Let $$\delta  \ll \epsilon_r \ll \epsilon_{r-1} \ll\ldots \ll \epsilon_1 \ll \epsilon_0=\epsilon \mbox{\quad and \quad} n_{t-1}\geq 1/\delta.$$

Let $G$ be a $K_t$-free spanning subgraph of $K_{n_1,\ldots,n_r}$ with $e(G)\geq f(n_1,\ldots,n_r,1,t)-\delta n_{t-1}^2$.
Let $G^0=G$ and define $G^i=G^{i-1}_{V_i}$ recursively.
Clearly, by Proposition \ref{neighbour copy keeps free} all of them are $K_t$-free and
\begin{equation}\label{eq 1 for weak stability}
e(G^r)\geq \ldots \geq e(G^0)=e(G)\geq f(n_1,\ldots,n_r,1,t)-\delta n_{t-1}^2.
\end{equation}

We first prove that $G^r$ has an $(X_r,\epsilon_r )$-stable $(t-1)$-partition $\mathcal{P}^r$  such that each $V_{i\leq r}$ is integral in $\mathcal{P}^r$.
The hierarchy of constants in the proof satisfy $\delta \ll \xi \ll \epsilon_r$.
Since $\delta$ is sufficiently small, by Lemma \ref{size gape} for $\delta$ there exists a $\xi\in [\delta,\sqrt[10^{r}]{\delta}]$, such that every set of $\mathcal{V}$ with size either less than $\xi n_{t-1}$ or at least $\sqrt[10]{\xi} n_{t-1}$.
Let $b$ be the minimum integer such that $n_b\leq \xi n_{t-1}$.
Let $X_r$ be the union of sets in $\mathcal{V}$ with size less than $\xi n_{t-1}$ and let $|X_r|=m_r= \sum_{i=b}^{r}n_i$.
Let $H_0=G^r-X_r$.
Define $H_{i+1}$ be the subgraph of $H_i$ induced by the neighbours of a vertex with maximum degree in $H_{i}$ recursively.
Since $G^r$ is $K_t$-free, we define graphs $H_1,\ldots,H_{s-1}$ with $s\leq t-1$.
Let $A_i=V(H_{i-1})-V(H_{i})$ for $i \in [s]$.
Clearly, $A_s$ is an independent set in $G^r$.

Note that every set of $\mathcal{V}_{X_r}$ is integral in $(A_1,\ldots,A_s)$ (each vertex in $V_i$ has common neighbors in $G^r$).
Thus\begin{equation}\label{eq 2 for weak stability}
e(H_0)\leq \sum\limits_{i=1}^{s-1} \left(\sum\limits_{v\in A_i}\frac{d_{A_i}(v)}{2}+d_{H_i}(v)\right)\leq \sum\limits_{1\leq i<j\leq s}|A_i||A_j|\leq f(n_1,\ldots,n_{b-1},1,s+1),
\end{equation}
where we set $A_0=\emptyset$.

If $s<t-1$, then by (\ref{eq 2 for weak stability}) and Proposition~\ref{pro:f(n,t)}, we have
\begin{eqnarray*}
 e(G^r) &\leq &   f(n_1,\ldots,n_{b-1},1,s+1)+m_r^2+m_r\sum_{i=1}^{b-1}n_i\\
    &\leq & f(n_1,\ldots,n_{b-1},1,t-1)+m_r^2+m_r\sum_{i=1}^{b-1}n_i\\
    &\leq & f(n_1,\ldots,n_{b-1},1,t)-n^2_{t-1}+  m_r^2+m_r\sum_{i=1}^{b-1}n_i\\
    &\leq & f(n_1,\ldots,n_{r},1,t)+ m_r \sum_{i=t-1}^{b-1}n_i -n^2_{t-1}+  m_r^2\\
     &\leq & f(n_1,\ldots,n_{r},1,t)-(1-\epsilon_r)n_{t-1}^2
\end{eqnarray*}
contradicting (\ref{eq 1 for weak stability}) (the fourth inequality holds by $f(n_1,\ldots,n_{r},1,t)\geq f(n_1,\ldots,n_{b-1},1,t)+m_r \sum_{i=1}^{t-2}n_i$ which can be easily proved by Theorem~\ref{extremal number 1}).
Therefore, we have $s=t-1$.

We can conclude that every vertex in $V(H_{i})$ for $i\in [s-1]$ is incident with all the vertices of $A_i$ in graph $H_{i-1}$.
Otherwise, suppose that $v_{j_1}\in V(H_{i})\cap V_{j_1}$ is not incident with some vertex $v_{j_2}\in A_i\cap V_{j_2}$.
Then there is no edge  between $V_{j_1}$ and $V_{j_2}$ in $G^r$.
Therefore, we have
\begin{equation*}
\begin{aligned}
\sum\limits_{v\in A_i}\frac{d_{A_i}(v)}{2}+d_{H_i}(v)&\leq  |A_i||V(H_i)|- \frac{ |V_{j_1}||V_{j_2}|}{2}\leq  |A_i||V(H_i)|-(\sqrt[10]{\xi}n_{t-1})^2.
\end{aligned}
\end{equation*}
It follows from (\ref{eq 2 for weak stability})  that  $e(H_0)\leq f(n_1,\ldots,n_{b-1},1,t)-\sqrt[5]{\xi}n_{t-1}^2$.
Since $m_r\leq r \xi n_{t-1}$, it is not hard to show that $e(G^r)\leq f(n_1,\ldots,n_{r},1,t)-\sqrt[4]{\xi}n_{t-1}^2$, a contradiction to (\ref{eq 1 for weak stability}).

Note that every vertex in $A_i$ has at most $|V(H_{i+1})|$ neighbours in $H_{i}$.
Thus $A_i$ is an independent set in $H_{i}$.
Therefore,  $(A_1,\ldots,A_{t-1})$ induces a complete $(t-1)$-partite subgraph in $H_0$, implying $G^r$ has an $(X_r,\epsilon_r)$-stable partition $(A_1,\ldots,A_{t-1})$ (recall that each part of $\mathcal{V}_{X_r}$ is integral in $(A_1,\ldots,A_{t-1})$).

Now we will show that if $G^k$ has an $(X_k,\epsilon_{k} )$-stable $(t-1)$-partition $\mathcal{P}^k$ and each $V_{j\leq k}\setminus X_k$ is integral in $\mathcal{P}^k_{X_k}$, then $G^{k-1}$ has an $(X_{k-1},\epsilon_{k-1}  )$-stable $(t-1)$-partition $\mathcal{P}^{k-1}$ with $X_{k}\subseteq X_{k-1}$.
Moreover, each $V_{j\leq k-1}\setminus X_{k-1}$ is integral in $\mathcal{P}^{k-1}_{X_{k-1}}$.

Let $\widetilde{G}^k=G^k-X_k$, $\mathcal{P}^k_{X_k}=(P^k_1,\ldots,P^k_{t-1})$ and $\widetilde{V}_i=V_i\setminus X_k$ for $i\in[r]$ with $\widetilde{V}_j\subseteq P^k_{i_j}$ for $j\in[k]$ (since each $V_{j\leq k}\setminus X_k$ is integral in $\mathcal{P}^k_{X_k}$).
Since $\mathcal{P}^k_{X_k}$ is $(X_k,\epsilon_{k})$-stable,
\begin{equation*}
d_{\widetilde{G}^k}(a)\geq  \sum_{i\neq i_k} |P_i^k|- \epsilon_{k} n_{t-1} \mbox{ for each } a\in  \widetilde{V}_k.
\end{equation*}
Let $Z_1$ be the vertices of $\widetilde{V}_k$ with degree less than $\sum_{i\neq i_k}|P_i^k|-\sqrt{\epsilon_{k}}n_{t-1}$ in $G^{k-1}$.
Thus, for each vertex $a\in\widetilde{V}_k \setminus Z_1$ in $G^{k-1}$, we have
\begin{equation}\label{minimal degree}
d_{G^{k-1}}(a)\geq  \sum_{i\neq i_k} |P_i^k|-  \sqrt{\epsilon_{k}} n_{t-1}.
\end{equation}
By our construction, we have
\begin{eqnarray*}
  e(G^k)&\geq &  e(G^{k-1})+|Z_1|\left(\sum_{i\neq  i_k } |P_i^k|- \epsilon_{k} n_{t-1}- \sum_{i\neq  i_k}|P_i^k|+\sqrt{\epsilon_{k}}n_{t-1}\right)  \\
    &\geq & e(G^{k-1})+|Z_1| (\sqrt{\epsilon_{k}}-  \epsilon_{k-1})n_{t-1}.
\end{eqnarray*}
implying $|Z_1|\leq   \sqrt{\epsilon_{k}} n_{t-1}$, as otherwise  by (\ref{eq 1 for weak stability}) we have $e(G^k)>  f(n_1,\ldots,n_r,1,t)$, hence $G^k$ contains a copy of $K_t$, a contradiction.

Clearly, the vertices in $\widetilde{V}_j$ have same neighbours in $G^{k-1}$ for $j \leq k-1$.
For $j \geq k+1$, similar as the above proof (the proof for the bound of $Z_1$), there exists a set of vertices $Z_2$ with size at most $\sqrt{\epsilon_{k}} n_{t-1}$  such that for each vertex $a\in\widetilde{V}_j \setminus Z_2$ in $G^{k-1}$, we have
\begin{equation}\label{minimal degree 1}
d_{G^{k-1}}(a)\geq  \sum_{i\neq i_j} |P_{i}^k|-  \sqrt{\epsilon_{k}}   n_{t-1}.
\end{equation}
Let $\mathcal{\widetilde{V}}=(\widetilde{V}_1,\ldots,\widetilde{V}_r)$.
By Proposition \ref{size gape}, for $\sqrt[4]{\epsilon_{k}}$ there exists $\xi$ with a $\xi\in [\sqrt[4]{\epsilon_{k}},\sqrt[4 \times 10^{(t-1)r}]{\epsilon_{k}}]$, such that every set of $(\mathcal{P}^k\wedge \mathcal{\widetilde{V}})_{X_k \cup  Z_1 \cup Z_2}$ with size either less than $\xi n_{t-1}$ or more than $\sqrt[10]{\xi} n_{t-1}$.
Let $Z_3$ be the union of sets in $(\mathcal{P}^k\wedge \mathcal{\widetilde{V}})_{X_k \cup  Z_1 \cup Z_2}$ with size less than $\xi n_{t-1}$.
Let $X_{k-1}^1=X_{k}\cup Z_1\cup Z_2\cup Z_3$.
Thus $|X_{k-1}^1|\leq 4tr\xi n_{t-1}$.

Let $S=V_{k}\cup X_{k-1}^1$.
Now we try to construct the desired $(t-1)$-partition of $G^{k-1}$ by distributing vertices in $V_{k}\setminus  X_{k-1}^1$ to each part of $\mathcal{P}^k_{S}=(\widetilde{P}_1,\ldots,\widetilde{P}_{t-1})$.
We partition $V_k\setminus  X_{k-1}^1$ into $B_1\cup \ldots\cup B_{t-1}$ such that if $v\in B_i$ then $v$ has minimum number of neighbours in $\widetilde{P}_i$.

\medskip

\noindent{\bf Claim 1.} There is no edge between $B_i$ and $\widetilde{P}_i$ for $i\in [t-1]$.

\medskip

\begin{proof}
Clearly, we may suppose that $|\widetilde{P}_i|\geq \sqrt[10]{\xi}n_{t-1}$, as otherwise there is nothing to show.
Let $a\in B_i$.
Due to (\ref{minimal degree}) and $\mathcal{P}^k$ is $\sqrt{\xi} n_{t-1}$-stable to $\mathcal{V}$, by Proposition~\ref{lemma for Q} there exists an $m(a)$ such that $a$ has more than $(\sqrt[4]{ \epsilon_k}/4)n_{t-1}$ vertices of the integral part of every $\widetilde{P}_{i\neq m(a)}$.
If there is an edge between  $a$ and $\widetilde{P}_i$, then by the definition of $B_i$, there is an edge $ab$ with $b \in\widetilde{P}_{m(a)}$.
Since $\mathcal{P}^k_{X_k}$ is $(X_k,\epsilon_{k})$-stable, we can construct a copy of $K_t$ in $G^{k-1}$ with vertices $a,b$ and their common neighbours in  the integral parts of $\widetilde{P}_{i\neq m(a)}$'s, a contradiction.
Hence we finish the proof of Claim 1.
\end{proof}

We construct the $(t-1)$-partition $\mathcal{P}_{X_{k-1}^1}^{k-1}$ with each part $P^{k-1}_j=\widetilde{P}_j\cup B_j$.
By Claim 1, each part of $\mathcal{P}^{k-1}_{X_{k-1}^1}$ is an independent set in $G^{k-1}$.
Moreover, since $\mathcal{P}^k_{X_k}$ is $(X_k,\epsilon_{k})$-stable, the size of $P^k_{i_k} \setminus V_{k}$ is no more than the size of the integral part of other class of $\mathcal{P}^k_{X_k}$.

By Proposition \ref{size gape} for $\xi$ there exists a $\zeta$ with $\zeta\in [{\xi},\sqrt[10^{(t-1)r}]{\xi}]$ such that every set of $(\mathcal{P}^{k-1}\wedge \mathcal{V})_{X_{k-1}^1}$ with size either less than $\zeta n_{t-1}$ or more than $\sqrt[10]{\zeta} n_{t-1}$.
Let $Z_4$ be the union of sets in $(\mathcal{P}^{k-1}\wedge \mathcal{V})_{X_{k-1}^1}$ with size less than $\zeta n_{t-1}$ and $X_{k-1}^2=X_{k-1}^1\cup Z_4$.
Thus $|X_{k-1}^2|\leq 6tr\zeta n_{t-1}$.

We  define $\widehat{V}_i=V_i\setminus X_{k-1}^2,\widehat{P}_j=P^{k-1}_j\setminus X_{k-1}^2$ and let $X_{i,j}=\widehat{V}_i\cap \widehat{P}_j$ for $i\in [r],j\in [t-1]$.
Now we will show that $\mathcal{P}^{k-1}_{X^2_{k-1}}$ is $\sqrt[8]{\zeta}n_{t-1}$-stable to $\mathcal{V}$.
The proofs of the following claims are quite similar to the proof in Theorem~\ref{extremal (t-1)-partition}.

\medskip

\noindent{\bf Claim 2.} If $\widehat{V}_i$ is partial in distinct $\widehat{P}_s,\widehat{P}_\ell$, then $\big||\widehat{P}_s\setminus \widehat{V}_i|-|\widehat{P}_\ell\setminus \widehat{V}_i|\big |\leq \sqrt{\zeta} n_{t-1}$.

\medskip

\begin{proof}
Without loss of generality, suppose for a contradiction that $\widehat{V}_1$ is partial in parts $\widehat{P}_1,\widehat{P}_2$ with $|\widehat{P}_2\setminus \widehat{V}_1|-|\widehat{P}_1\setminus \widehat{V}_1|> \sqrt{\zeta} n_{t-1}$.
Let $S=\bigcup\nolimits_{i=2}^{t-1} (\widehat{P}_i\setminus \widehat{V}_1)$.
Note that $|X_{1,1}|\geq \sqrt[10]{\zeta}n_{t-1}$.
Thus  by \eqref{eq 1 for weak stability} and Proposition \ref{neighbour copy keeps free}, we know that $H=G^{k-1}_{X_{1,1}\to S}$ is a $K_t$-free graph with
\begin{eqnarray*}
e(H)&\geq& e(G^{k-1})+ \sqrt{\zeta}n_{t-1}|X_{1,1}|\\
&\geq& f(n_1,\ldots,n_r,1,t)-\delta n_{t-1}^2+ \sqrt{\zeta} n_{t-1} \sqrt[10]{\zeta} n_{t-1}\\
&>& f(n_1,\ldots,n_r,1,t),
\end{eqnarray*}
 a contradiction.
We finish the proof of Claim 2.
\end{proof}

\medskip

\noindent{\bf Claim 3.} $\mathcal{P}^{k-1}_{X_{k-1}^2}$ is $1$-partial to $\mathcal{V}_{X_{k-1}^2}$.

\medskip

\begin{proof}
Without loss of generality, suppose for a contradiction that both $\widehat{V}_1,\widehat{V}_2$ are partial in $\widehat{P}_1$.
Let $\widehat{V}_1$ be partial in $\widehat{P}_{s}$ and $\widehat{V}_2$ be partial in $\widehat{P}_{\ell}$ (it is possible that $s=\ell$).
Since $\widehat{V}_2$ is partial in $\widehat{P}_1,\widehat{P}_{\ell}$, it follows from Claim 2 that $|\widehat{P}_1\setminus \widehat{V}_2|-|\widehat{P}_{\ell}\setminus \widehat{V}_2|\geq -\sqrt{\zeta} n_{t-1}$.
Let $S_1=\bigcup\limits_{i\neq \ell} (\widehat{P}_i\setminus \widehat{V}_2)$.
Thus $H^1=G^{k-1}_{X_{2,1}\to S_1}$ is a $K_t$-free graph with
$$e(H^1)\geq e(G^{k-1})+ (|\widehat{P}_1\setminus \widehat{V}_2|-|\widehat{P}_{\ell}\setminus \widehat{V}_2|-|X_{k-2}^2|)|X_{2,1}|
\geq e(G^{k-1}) -2\sqrt{\zeta}n_{t-1}|X_{2,1}|.$$
Let $S_2=\bigcup\limits_{i=2}^{t-2}(\widehat{P}_i\setminus \widehat{V}_1)\cup X_{2,1}$.
Since $\widehat{V}_1$ is partial in $\widehat{P}_1,\widehat{P}_{s}$,   Claim 2 implies $|\widehat{P}_1\setminus X_{1,1}|-|\widehat{P}_{s}\setminus X_{1,s}|\leq \sqrt{\zeta} n_{t-1}$.
Note that $|X_{2,1}|,|X_{1,s}|\geq \sqrt[10]{\zeta}n_{t-1}$.
Thus $H^2=H^1_{X_{1,s}\to S_2}$ is a $K_t$-free graph with
\begin{eqnarray*}
e(H^2)&\geq& e(H^1)+(|X_{2,1}|-\sqrt{\zeta}n_{t-1}-|X_{k-2}^2|)|X_{1,s}|\\
&\geq& e(H^1)+\dfrac{1}{2}|X_{2,1}||X_{1,s}|\\
& \geq& e(G^{k-1})+\dfrac{1}{2}\sqrt[5]{\zeta}n_{t-1}^2\\
&>&f(n_1,\ldots,n_r,1,t),
\end{eqnarray*}
a contradiction.
Thus $\mathcal{P}^{k-1}_{X_{k-1}^2}$ is $1$-partial to $\mathcal{V}_{X_{k-1}^2}$.
\end{proof}

By Claim 3 we know that if $\widehat{V}_i$ is partial in $\widehat{P}_j$, then $\widehat{P}_j\setminus \widehat{V}_i$ is the integral part of $\widehat{P}_j$.

\medskip

\noindent{\bf Claim 4.} If $\widehat{P}_{s},\widehat{P}_{\ell}$ are distinct partial classes of $\mathcal{P}^{k-1}_{X_{k-1}^2}$, then the integral  parts  of $\widehat{P}_{s}$ and $\widehat{P}_{\ell}$ are less than $\sqrt[5]{\zeta} n_{t-1}$ difference in sizes.

\medskip

\begin{proof}
In Claim 2, we have already showed that Claim 4 holds when $\widehat{V}_i$ is partial in both $\widehat{P}_s,\widehat{P}_\ell$.
Without loss of generality, we only need to show if $\widehat{V}_1$ is partial in $\widehat{P}_1$ and $\widehat{V}_2$ is partial in $\widehat{P}_2$ then $\big||\widehat{P}_1\setminus \widehat{V}_1|-|\widehat{P}_2\setminus \widehat{V}_2|\big|\leq \sqrt[5]{\zeta} n_{t-1}.$
Without loss of generality, suppose that $\widehat{V}_2$ is also partial in $\widehat{P}_3$ and let $a\in X_{2,2}$ be a vertex with maximum degree among $X_{2,2}\cup X_{2,3}$ in graph $G^{k-1}$.
Thus $H^1=G^{k-1}_{X_{2,3}\to N_{G^{k-1}(a)}}$ keeps $K_t$-free with $e(H^1)\geq e(G^{k-1})$.
Let $S=\bigcup\limits_{i\neq 3} (\widehat{P}_i\setminus \widehat{V}_1)\cup X_{2,3} $.
Note $H^2=H^1_{X_{1,1}\to S}$ is a $K_t$-free graph.
Thus
$$f(n_1,\ldots,n_r,1,t)\geq e(H^2)\geq e(G^{k-1})+(|\widehat{P}_1\setminus \widehat{V}_1|-|\widehat{P}_3\setminus \widehat{V}_2|-|X_{k-2}^2|)|X_{1,1}|,$$
which implies $|\widehat{P}_1\setminus \widehat{V}_1|-|\widehat{P}_3\setminus \widehat{V}_2|\leq \sqrt[4]{\zeta}n_{t-1}$.
By Claim 2 we know $|\widehat{P}_3\setminus \widehat{V}_2|-|\widehat{P}_2\setminus \widehat{V}_2|\leq \sqrt{\zeta}n_{t-1}$.
Thus $|\widehat{P}_1\setminus \widehat{V}_1|-|\widehat{P}_2\setminus \widehat{V}_2|\leq \sqrt[5]{\zeta}n_{t-1}$.
By symmetry of $\widehat{P}_1,\widehat{P}_2$, we have $|\widehat{P}_2\setminus \widehat{V}_2|-|\widehat{P}_1\setminus \widehat{V}_1|\leq \sqrt[5]{\zeta}n_{t-1}$ and finish the proof of Claim 4.
\end{proof}

\noindent{\bf Claim 5.} If $\widehat{V}_i$ is partial in $\widehat{P}_j$, then $|\widehat{P}_j\setminus \widehat{V}_i|\geq |\widehat{V}_i|-\sqrt[7]{\zeta}n_{t-1}$.

\medskip

\begin{proof}
Without out loss of generality, suppose that $\widehat{V}_1$ is partial in $\widehat{P}_1,\widehat{P}_2,\ldots,\widehat{P}_s$ with $|\widehat{P}_1\setminus \widehat{V}_1|\leq |\widehat{V}_1|-\sqrt[7]{\zeta}n_{t-1}$.

In graph $G^{k-1}$, we turn the neighbours of every vertex of $\widehat{V}_1$ to $\bigcup_{j=1}^{t-1} (\widehat{P}_j\setminus \widehat{V}_1)$ and delete the edges between $\widehat{P}_1\setminus \widehat{V}_1$ and $\widehat{P}_2\setminus \widehat{V}_1$.
Denote the resulting graph by $J$.
Note that $J$ is a $K_t$-free graph with
\begin{displaymath}
\begin{aligned}
e(J)&\geq e(G^{k-1})+(|\widehat{P}_1\setminus \widehat{V}_1|-2\sqrt[5]{\zeta}n_{t-1})|\widehat{V}_1|-|\widehat{P}_1\setminus \widehat{V}_1||\widehat{P}_2\setminus \widehat{V}_1|\\
&\geq e(G^{k-1})+|\widehat{P}_1\setminus \widehat{V}_1|( \sqrt[7]{\zeta}n_{t-1}-3\sqrt[5]{\zeta}n_{t-1})-2\sqrt[7]{\zeta}\cdot \sqrt[5]{\zeta} n_{t-1}^2\\
&\geq e(G^{k-1})+\sqrt{\zeta}n_{t-1}^2>f(n_1,\ldots,n_r,1,t),
\end{aligned}
\end{displaymath}
a contradiction.
The proof of Claim 5 is complete.
\end{proof}

\noindent{\bf Claim 6.} If $\widehat{V}_i$ is partial in $\widehat{P}_\ell$ and $\widehat{P}_s$ is an integral class, then $|\widehat{P}_j\setminus \widehat{V}_i|\leq |\widehat{P}_s|+\sqrt[7]{\zeta}n_{t-1}$.

\medskip
\begin{proof}
Let $S=\bigcup\limits_{\iota\neq s}(\widehat{P}_\iota\setminus \widehat{V}_i)$.
Otherwise, the graph $H=G^{k-1}_{X_{i,\ell}\to S}$ is $K_t$-free graph with more than $e(G^{k-1})+(\sqrt[7]{\zeta}-\sqrt{\zeta})\sqrt[10]{\zeta}n_{t-1}^2> f(n_1,\ldots,n_r,1,t)$ edges, a contradiction.
Therefore, we complete the proof of Claim 6.
\end{proof}

\noindent{\bf Claim 7.} If $\widehat{V}_i$ is integral in $\widehat{P}_\ell$ and $\widehat{V}_j$ is partial in $\widehat{P}_s$, then $|\widehat{P}_\ell \setminus \widehat{V}_i| \leq |\widehat{P}_s \setminus \widehat{V}_j|+\sqrt[7]{\zeta}n_{t-1}$.

\medskip
\begin{proof}
Let $\widehat{V}_j$ be partial in $\widehat{P}_{s^\prime}$.
We first suppose $s^\prime \neq \ell$.
Without loss of generality, let $a\in X_{j,s}$ be a vertex with maximum degree among set $X_{j,s}\cup X_{j,s^\prime}$ in graph $G^{k-1}$.
Thus $H^1=G^{k-1}_{X_{j,s^\prime}\to N_{G^{k-1}(a)}}$ keeps $K_t$-free with $e(H^1)\geq e(G^{k-1})$.
Let $S=\bigcup_{\iota\neq s^\prime}(\widehat{P}_\iota\setminus \widehat{V}_i)\cup X_{j,s^\prime}$ and $H^2=H^1_{X_{i,\ell}\to S}$.
Note that $H^2$ is a $K_t$-free graph.
We have $f(n_1,\ldots,n_r,1,t)\geq e(H^2)\geq e(G^{k-1})+(|\widehat{P}_\ell \setminus \widehat{V}_i|- |\widehat{P}_{s^\prime} \setminus \widehat{V}_j|-|X_{k-2}^2|)|X_{i,\ell}|$, implying $|\widehat{P}_\ell \setminus \widehat{V}_i|- |\widehat{P}_{s^\prime} \setminus \widehat{V}_j|\leq \sqrt[5]{\zeta}n_{t-1}$.
Since $\big| |\widehat{P}_s \setminus \widehat{V}_j|- |\widehat{P}_{s^\prime} \setminus \widehat{V}_j| \big| \leq \sqrt{\zeta}n_{t-1}$, we have $|\widehat{P}_\ell \setminus \widehat{V}_i| \leq |\widehat{P}_s \setminus \widehat{V}_j|+\sqrt[7]{\zeta}n_{t-1}$.

Suppose $s^\prime =\ell$.
Let $S_1=\bigcup_{\iota\neq \ell} \widehat{P}_\iota\setminus \widehat{V}_j$ and $H^1=G^{k-1}_{X_{j,s}\to S_1}$.
Let $S_2=\bigcup_{\iota\neq s} (\widehat{P}_\iota\setminus \widehat{V}_i)\cup X_{j,s}$ and $H^2=H^1_{\widehat{V}_i\to S_2}$.
Clearly, $H^1$ and $H^2$ are both $K_t$-free.
Note that $|\widehat{V}_i|\geq \sqrt[10]{\zeta}n_{t-1}$.
By Claim 2 we have $|\widehat{P}_s\setminus \widehat{V}_j|-|\widehat{P}_\ell\setminus \widehat{V}_j|\geq -\sqrt{\zeta}n_{t-1}$, hence
\begin{equation}
\begin{aligned}
f(n_1,\ldots,n_r,1,t)&\geq e(H^2)\geq e(H^1)+(|\widehat{P}_\ell\setminus \widehat{V}_i|-|\widehat{P}_s\setminus \widehat{V}_i|+|X_{j,s}|-|X_{k-1}^2|)|\widehat{V}_i|\\
&\geq e(G^{k-1})+(|\widehat{P}_s\setminus \widehat{V}_j|-|\widehat{P}_\ell\setminus \widehat{V}_j|-|X^2_{k-1}|)|X_{j,s}|\\
&+(|\widehat{P}_\ell\setminus \widehat{V}_i|-|\widehat{P}_s\setminus \widehat{V}_i|+|X_{j,s}|-|X_{k-1}^2|)|\widehat{V}_i|\\
&\geq e(G^{k-1})+(|\widehat{P}_\ell\setminus \widehat{V}_i|-|\widehat{P}_s\setminus \widehat{V}_i|)|\widehat{V}_i|.
\end{aligned}
\end{equation}
Therefore, we have $|\widehat{P}_\ell\setminus \widehat{V}_i|\leq |\widehat{P}_s\setminus \widehat{V}_j|+\sqrt[7]{\zeta}n_{t-1}$, and we finish the proof of Claim 7.
\end{proof}

By Claims 3-7, $\mathcal{P}^{k-1}_{X^2_{k-1}}$ is $\sqrt[8]{\zeta}n_{t-1}$-stable to $\mathcal{V}$.
By Lemma \ref{remove vertices to make it stable}, we can remove at most $2tr\sqrt[7]{\zeta}n_{t-1}$ vertices, say $Z_5$, to ensure that $(\mathcal{P} \wedge  \mathcal{V})_{X_{k-1}^2\cup Z_5}$  satisfies the following.
\begin{itemize}
\item  the sizes of integral parts of partial classes are equal to each other,
\item  the size of the integral part of any class is at least  the size of  any partial class and
\item  after removing any integral $V_i$, the size of the resulting set is at most the size of the integral part of any other class.
\end{itemize}
Let $X_{k-1}=X_{k-1}^2\cup Z_5$.
Thus $|X_{k-1}|\leq \sqrt[8]{\zeta}n_{t-1}\leq \epsilon_{k-1} n_{t-1}$.
Combining (\ref{minimal degree}) and (\ref{minimal degree 1}),  we can see that the partition $\mathcal{P}^{k-1}$ is an $(X_{k-1},\epsilon_{k-1})$-stable partition of $G$.
Moreover, every $V_{i\leq k-1}\setminus X_{k-1}$ keeps integral in $\mathcal{P}^{k-1}_{X_{k-1}}$.
Furthermore, similar as the proof of Proposition~\ref{larger than n_{t-1}}, the size of integral part of each class of $\mathcal{P}_{X_{k-1}}$  is at least $(1-\epsilon_{k-1}) n_{t-1}$.

Thus recursively, we obtain an $(X_0,\epsilon)$-stable partition $\mathcal{P}^0$ of $G$, and  we finish the proof of Theorem~\ref{weak stability n case}.
\end{proof}

Now we start to prove Theorem \ref{weak stability}.

\medskip

\noindent{\bf Proof of Theorem~\ref{weak stability}.}
Note that each graph with chromatic number $t$ is a subgraph of a $t$-partite  graph.
We only need to prove that the theorem holds for $K_t^s$ where $s$ is an integer depending on $F$.

Let $G$ be a  $K_t^s$-free $r$-partite graph with parts $\mathcal{V}=(V_1,\ldots,V_r)$.
Let  $0<\delta\ll \delta_1 \ll\epsilon<1$, where $\delta$ and $\delta_1$   will be determined later by $\epsilon$, Lemma \ref{removal lemma} and Theorem \ref{weak stability n case}.
Let $n \geq 1/\delta $.
Suppose that
\begin{equation}\label{eq for weak}
e(G)\geq f(n_1,\ldots,n_t,1,t)-\delta n^2_{t-1}.
\end{equation}
We shall show that $G$ has an $(X,\epsilon )$-stable $(t-1)$-partition $\mathcal{P}$ after removing $ \epsilon n_{t-1}^2$ edges.

Let $\tau=\tau(n_1,\ldots,n_r,t,tr)$.
We first discuss the case $\tau=0$.
Since $\tau=0$, we have $n_{t-1}\geq (1/tr)^{4t}n$.
For  $\alpha=(1/tr)^{8t}\delta_1$, by Lemma \ref{removal lemma}, $G$ contains a $K_t$-free graph $\widetilde{G}$ with $e(\widetilde{G})\geq e(G)-(1/tr)^{8t}\delta_1 n^2 \geq e(G)-\delta_1 n_{t-1}^2$ edges.
Hence by (\ref{eq for weak}), the resulting graph $\widetilde{G}$ is $K_t$-free with $e(\widetilde{G})\geq  f(n_1,\ldots,n_t,1,t)- 2\delta_1 n^2_{t-1}$.
Therefore, by Theorem \ref{weak stability n case}, $\widetilde{G}$ has an $(X,\epsilon)$-stable $(t-1)$-partition $\mathcal{P}$ and each class of $\mathcal{P}$ is larger than $(1-\epsilon) n_{t-1}$.
Thus we finish the proof of the theorem in this case.

Then we discuss the case when $\tau\geq 1$.
Let $U_L$ be the union of large sets $V_1,\ldots,V_{\tau}$ and $U_S=V(K)-V_L$ to be the union of rest small sets.

If $\tau\geq 1$, then $\tau\leq t-2$ (recall the definition of $\tau$).
Clearly, we have
\begin{equation}\label{1.4 Eq 1}
f(n_1,\ldots,n_r,1,t)\geq e(K_{n_1,\ldots,n_\tau})\geq e(K_{n_1,\ldots,n_r})-|U_S|^2\geq e(K_{n_1,\ldots,n_r})-(r n_{\tau+1})^2.
\end{equation}
Let $Z_i^1$ be the vertices of large set $V_i$ with degree less than $n-|V_i|-2rn_{\tau+1}$.
Thus $e(G)\leq e(K_{n_1,\ldots,n_r})-2rn_{\tau+1}|Z_i^1|$.
Hence by (\ref{eq for weak}) and (\ref{1.4 Eq 1}), we have $2rn_{\tau+1}|Z_i^1|\leq 2(rn_{\tau+1})^2$ and thus $|Z_i^1|\leq rn_{\tau+1}\leq r(1/tr)^4n_\tau\leq (1/tr)^3n_\tau$.
Note that every vertex of $V_i$ except $Z_i^1$ is adjacent to all but at most $2rn_{\tau+1}\leq 2rn_\tau /(tr)^4$ vertices in other classes of  $K_{n_1,\ldots,n_\tau}$.
Therefore, we can easily find a copy of $K_\tau^s$ in the subgraph of $G$ induced by the union of any $n_\tau/2$ vertices from each set of $V_1,\ldots,V_\tau$.

Now we prove $G[U_S]$ can be $K^s_{t-\tau}$-free by only removing $\delta_1 n_{\tau+1}^2$ edges.
Suppose $G[U_S]$ contains a copy of $K^s_{t-\tau}$.
Then there is at least one vertex of $K^s_{t-\tau}$ having neighbours less than $|U_L|-2tr n_{\tau+1}$ neighbours in $U_L$, as otherwise all vertices of $K^s_{t-\tau}$ will have more than $n_i-2(tr)^2 n_{\tau+1} \geq  n_\tau/2$ common neighbours in each set of $V_1,\ldots,V_{\tau}$.
Thus we can extend $K^s_{t-\tau}$ to $K_t^s$ with some $s$ vertices from each $V_i$, a contradiction to that $G$ is $K_{t}^s$-free.
Therefore, at least one vertex $v_1$ of $K^s_{t-\tau}$ has degree at most $|U_L|-2tr n_{\tau+1}$ in $U_L$, implying $d_G(v_1)\leq |U_L|-tr n_{\tau+1}$.

Let $H^1=G_{\{v_1\}\to U_L}$.
Thus $e(H^1)\geq e(G)+|U_L|-d_G(v_1)\geq e(G)+trn_{\tau+1}$.
Since $G$ is $K_t^s$-free and $v_1$ is not adjacent any copy of $K_{t-1}$ in $H^1$, $H^1$ is still $K_t^s$-free.

We repeat this process recursively and obtain graph $H^{C}$ such that $H^C[U_S]$ is $K_{t-\tau}^s$-free.
By Lemma \ref{removal lemma} $H^C[U_S]$ contains a $K_{t-\tau}$-free subgraph $H_S$ with $e(H_S)\geq e(H^C[U_S])-(1/tr)^{8t}\delta_1  n_{\tau+1}^2$.
Thus by Theorem~\ref{extremal number 1}, we have
\begin{equation}\label{eq 7}
e(H^C[U_S])\leq  f(n_{\tau+1},\ldots,n_r,1,t-\tau)+(1/tr)^{8t}\delta_1  n_{\tau+1}^2
\end{equation}
It follows from Lemma \ref{remove big set} and \eqref{eq for weak} that
\begin{equation}\label{eq 8}
\begin{aligned}
e(H^C[U_S])&\geq e(H^C)-\sum\limits_{1\leq i<j\leq \tau}n_i n_j-\sum\limits_{i=1}^{\tau}n_i \sum\limits_{i=\tau+1 }^{r}n_i\\
&\geq e(G)+Ctrn_{\tau+1}-\sum\limits_{1\leq i<j\leq \tau}n_i n_j-\sum\limits_{i=1}^{\tau}n_i \sum\limits_{i=\tau+1 }^{r}n_i\\
&\geq  f(n_{\tau+1},\ldots,n_r,1,t-\tau)+Ctrn_{\tau+1}-\delta n_{t-1}^2.
\end{aligned}
\end{equation}
Therefore, by \eqref{eq 7} and \eqref{eq 8} we have $Ctrn_{\tau+1}\leq ( 1/tr )^{8t}\delta_1  n_{\tau+1}^2+\delta n_{t-1}^2 \leq ( 1/tr )^{8t}tr\delta_1 n_{\tau+1}^2$.
Thus we only need to remove $C |U_S|\leq (1/tr)^{8t}tr\delta_1 n_{\tau+1}^2\leq \delta_1 n_{t-1}^2$ edges to make $G[U_S]$ being $K^s_{t-\tau}$-free.

We denote the new graph by $\widetilde{G}$ after removing these $\delta_1 n_{t-1}^2$ edges in $G$.
Thus
\begin{equation*}
\begin{aligned}
e(\widetilde{G}[U_S])
&\geq e(G[U_S])-\delta_1 n_{t-1}^2\\
&\geq e(G)-\sum\limits_{1\leq i<j\leq \tau}n_i n_j-\sum\limits_{i=1}^{\tau}n_i \sum\limits_{i=\tau+1 }^{r}n_i-\delta_1 n_{t-1}^2\\
&\geq f(n_1,\ldots,n_r,1,t)-\sum\limits_{1\leq i<j\leq \tau}n_i n_j-\sum\limits_{i=1}^{\tau}n_i \sum\limits_{i=\tau+1 }^{r}n_i-2\delta_1 n_{t-1}^2\\
&\geq f(n_{\tau+1},\ldots,n_r,1,t-\tau)-2\delta_1 n_{t-1}^2.
\end{aligned}
\end{equation*}
Hence by Theorem \ref{weak stability n case} there exists an $(X_S,\epsilon/2)$-stable $(t-\tau-1)$-partition $\mathcal{P}$ of $\widetilde{G}[U_S]$ and each class of $\mathcal{P}_{X_S}$ is larger than $(1-\epsilon/2)n_{t-1}$.

Let $Z_i^2$ be the vertices of large set $V_i$ with degree less than $n-|V_i|-2\sqrt{\delta}n_{t-1}$.
From \eqref{eq for weak}, we have
$$e(\widetilde{G})-e(\widetilde{G}[U_S])\geq \sum\limits_{1\leq i<j\leq \tau}n_i n_j-\sum\limits_{i=1}^{\tau}n_i \sum\limits_{i=\tau+1 }^{r}n_i-\delta n_{t-1}^2.$$
Thus we have $|Z^2_i|\leq 2\sqrt{\delta}n_{t-1}$.
Let $X=X_S\cup Z_1^2\cup \ldots\cup Z_\tau^2$.
Therefore, the partition $\mathcal{P}$ is an $(X, \epsilon)$-stable $(t-1)$-partition of $G$ after removing $\epsilon n_{t-1}^2$ edges with $|X|\leq (2r\sqrt{\delta}+\epsilon/2) n_{t-1} \leq \epsilon n_{t-1}$ and each class of $\mathcal{P}$ is larger than $(1-\epsilon)n_{t-1}$.
We finish the proof of Theorem~\ref{weak stability}.\hfill$\square$ \medskip

\noindent {\bf Proof of Theorem \ref{strong stability}}.
Let $F$ be a graph with chromatic number $t\geq 3$ and $G \subseteq K_{n_1,\ldots,n_r}$ be an extremal graph for $F$.
Let $\epsilon \gg \delta_3  \gg \delta_2 \gg \delta_1>0$ and $n_{t-1}\geq 1 / \delta_1$.

By Theorem~\ref{extremal number 1}, there is an extremal graph for $K_{t}$ which is $(t-1)$-partite, implying $e(G)\geq f(n_1,\ldots,n_r,1,t)$ (since  every $(t-1)$-partite graph is $F$-free).
Since $\delta_2 \gg \delta_1$,   by Theorem \ref{weak stability}, after removing at most $\delta_2n_{t-1}^2$ edges, the resulting graph $\widetilde{G}$ has an $(X,\delta_2)$-stable $(t-1)$-partition $\mathcal{P}=(P_1,\ldots,P_{t-1})$.
Moreover, we partition $X$ into $X_i\subseteq P_i$ such that each vertex of $X$ has smallest number of neighbours in its own class of $\mathcal{P}$.

Let  $\epsilon \gg \delta^\prime_3  \gg \delta_2$.
By Proposition \ref{size gape} there exists an $\epsilon_1$ with $\epsilon_1\in [\delta^\prime_3,\sqrt[10^{(t-1)r}]{\delta^\prime_3}]$, such that every set of $(\mathcal{P}\wedge \mathcal{V})_{X}$ with size either less than $\epsilon_1 n_{t-1}$ or more than $\sqrt[10]{\epsilon_1} n_{t-1}$.
It is clear that $\epsilon_1 \ll \epsilon$.
Let $Z_i$ be subset of $P_i$ in $(\mathcal{P}\wedge \mathcal{V})_X$ with size less than $\epsilon_1 n_{t-1}$ and $Z=(\bigcup^{t-1}_{i=1}Z_i)  \cup X$.
Thus $|Z|\leq 2tr\epsilon_1 n_{t-1}$.

Let $\widetilde{V}_i=V_i\setminus  Z$, $\widetilde{P}_i=P_i\setminus Z$, $X_{i,j}=\widetilde{V}_i\cap \widetilde{P}_j$ and $\widetilde{\mathcal{V}}=(\widetilde{V}_1,\ldots,\widetilde{V}_r)$.

Let $v\in V_k$ and $P_{i_k}\setminus V_k$ be any smallest part among $\mathcal{P}_{V_k}$.
Let $\widetilde{V}_\iota$ be integral in $\widetilde{P}_{i_k}$.
Choose $B \subseteq\widetilde{V}_\iota$ with size $|F|$.
By Proposition \ref{neighbour copy keeps free}, $G_{\{v\} \to N_G[B]}$ is also $F$-free.
Since $G$ is the extremal graph, we have
\begin{equation}\label{minimum degree in F free extremal}
d_G(v)\geq  |N_{\widetilde{G}}[B]|  \geq \sum_{i\neq i_k}|P_i\setminus V_k|-\sqrt{\epsilon_1} n_{t-1}.
\end{equation}
If $V_k$ does not belong to $Z$, then since $\mathcal{P}_Z$ is $\sqrt{\epsilon_1}n_{t-1}$-stable to $\widetilde{\mathcal{V}}$ by \eqref{minimum degree in F free extremal} and Proposition~\ref{lemma for Q}, there exists an integer $m(v)\in [1,t-1]$ with $v$ adjacent to $\sqrt[6]{\epsilon_1}n_{t-1}$ vertices from the integral part of each $\widetilde{P}_{i\neq m(v)}$.

Suppose that $V_k \subseteq Z$.
Let $m$ be the minimum size of  integral parts of classes of $\mathcal{P}_Z$ and $m^\prime$ be the minimum size of partial parts of classes of $\mathcal{P}_Z$ ($\mathcal{P}_Z$ is 1-partial).
By (\ref{minimum degree in F free extremal}), we have $d_G(v)\geq n-m-m^\prime-3|Z|$.
Suppose that $v\in V_k$ is adjacent to less than $\sqrt[6]{\epsilon_1}n_{t-1}$ vertices from the integral part of $\widetilde{P}_{i_1}$ and $\widetilde{P}_{i_2}$.
Let  $X_{a_1,i_1}$  and $X_{a_2,i_2}$   be the possible partial parts of $\widetilde{P}_{i_1}$ and $\widetilde{P}_{i_2}$ (if $\widetilde{P}_{i}$ is integral, we set the partial part empty).
Since $\mathcal{P}=(P_1,\ldots,P_{t-1})$ is an $(X,\delta_2)$-stable $(t-1)$-partition, the integral part of $\widetilde{P}_{i}$ is at least $m-|Z|$ and $m^\prime\leq m/2+|Z|$.
Thus
$$d_G(v)\leq \sum_{i\neq i_1,i_2}|P_i|+|X_{a_1,i_1}|+|X_{a_2,i_2}|+3\sqrt[6]{\epsilon_1}n_{t-1}<  n-m-m^\prime-3|Z|, $$
a contradiction.
Thus there exists an integer $m(v)\in [1,t-1]$ such that $v$ is adjacent to $\sqrt[6]{\epsilon_1}n_{t-1}$ vertices from the integral part of each $\widetilde{P}_{i\neq m(v)}$.

Let $Y_i$ be the vertices of $P_i$ with more than $\epsilon n_{t-1}$ neighbours in their own part.
Let $c_{\epsilon,F}=tN(\sqrt[3]{\epsilon_1},\ldots,\sqrt[3]{\epsilon_1};1/2;|F|)$ and $Y=\bigcup_{i=1}^{t-1}Y_i$.
Thus every vertex $v$ of $Y$ is adjacent to $\sqrt[6]{\epsilon_1}n_{t-1}$ vertices from the integral part of each $\widetilde{P}_{i\neq m(v)}$ and $\sqrt[6]{\epsilon_1}n_{t-1}$ vertices from part $\widetilde{P}_{m(v)}$.
Since $\mathcal{P}$ is an $(X,\zeta )$-stable partition of $\widetilde{G}$,  we can pick $\sqrt[5]{\epsilon_1}n_{t-1}$ vertices $Q_i(v)$ from each $\widetilde{P}_{i}$ to form a copy of $K_{t-1}^{\sqrt[5]{\epsilon_1}n_{t-1}}$ whose vertices are all adjacent to $v$.

Let $\mathcal{Q}(v)=(Q_1(v),\ldots,Q_{t-1}(v))$.
For a set of vertices $X$, let  $$\bigcap_{v\in X} \mathcal{Q}(v)=\left(\bigcap_{v\in X}Q_1(v),\ldots,\bigcap_{v\in X}Q_{t-1}(v)\right).$$

Suppose that some $|Y_i|$ is larger than $c_{\epsilon,F}/t$.
Let $\tau=\tau(n_1,\ldots,n_r,t,r)$.
We divide the proof into the following two cases.

If $\tau=0$, then $n_{t-1}\geq (1/r)^{4t}n$.
Note that each part size of $\mathcal{Q}(a)$ is larger than $\sqrt[5]{\epsilon_1}n_{t-1}\geq \sqrt[4]{\epsilon_1}n$ for each vertex $a\in Y_i$.
Thus by Lemma \ref{selection lemma}, there exist $|F|$ vertices, denote it $C_F$, of $Y_i$ such that every part size of $\bigcap_{v\in C_F} \mathcal{Q}(v)$ is larger than $\epsilon_1 n_{t-1}/2$.
Hence $C_F$ and $\bigcap_{v\in C_F} \mathcal{Q}(v)$ can form a copy of  $K_t^{\epsilon_1 n_{t-1}/4}$, implying $F\subset G$, a contradiction.
Hence $|Y_i|\leq c_{\epsilon,F}/t$, implying $|Y|\leq c_{\epsilon,F}$.

Let $\tau\geq 1$.
We will show that $V_1,\ldots,V_\tau$ are parts of $\mathcal{P}$.
We will first show $\widetilde{V}_1,\ldots,\widetilde{V}_\tau$ are parts of $\mathcal{P}_Z$.
Suppose that part $\widetilde{P}_i$ of $\mathcal{P}_Z$ contains vertices of set in $\widetilde{V}_1,\ldots,\widetilde{V}_\tau$, let it be $\widetilde{V}_{i_1}$, and also contains vertices from other set of $\widetilde{\mathcal{V}}$, let it be $\widetilde{V}_{i_2}$.
If $\widetilde{V}_{i_1}$ is integral in $\widetilde{P}_i$ then $|\widetilde{P}_i\setminus \widetilde{V}_{i_2}|\geq |\widetilde{V}_{i_1}|$.
If $\widetilde{V}_{i_1}$ is partial in $\widetilde{P}_i$ then due to $\mathcal{P}$ is $Y$-stable of $G$ then $|\widetilde{P}_i\setminus \widetilde{V}_{i_1}|\geq |V_{i_1}|-|Y|$.
Thus by letting $k=i_1,i_2$ and $m=1,\ldots,t-1$ in (\ref{minimum degree in F free extremal}) we obtain each part of $\mathcal{P}_Z$ is larger than $ |\widetilde{V}_{i_1}|/2$, implying every part of $\mathcal{P}_Z$ contains at least one of $\widetilde{V}_1,\ldots,\widetilde{V}_\tau$, a contradiction due to $\tau\leq t-2$.
Thus let $\widetilde{P}_1=\widetilde{V}_1,\ldots,\widetilde{P}_\tau=\widetilde{V}_\tau$.
Due to each $Y_i$ has minimal neighbours in $\widetilde{P}_i$ and by (\ref{minimum degree in F free extremal}) trivially, we obtain $P_1=V_1,\ldots,P_\tau=V_\tau$.

Suppose that $|Y_i|\geq c_{\epsilon,F}/t$ for some $Y_i$ .
Since vertex in every $P_{i\leq \tau}$ has no neighbour in its own part, we only consider $Y_{i>\tau}$.
Note $n_{t-1}\geq (1/r)^{4t}n_{\tau+1}$.
Thus every $P_{i>\tau}$ has $|P_i|\leq rn_{\tau+1}\leq r^{5t}n_{t-1}$.
For each vertex $a\in Y_i$, the size of each part of $\mathcal{Q}(a)$ in  $P_{i>\tau}$ is larger than $\sqrt[5]{\epsilon_1}n_{t-1}\geq \sqrt[4]{\epsilon_1}|P_i|$ .
As the proof in the case $\tau=0$, there is a copy of $K_{t-\tau}^{\epsilon_1/4 n_{t-1}}$ in $P_{i>\tau}$'s.
Note that each vertex of $K_{t-\tau-1}^{|F|}$ have at least $|V_{i\leq \tau}|-\delta_2 n_{t-1}$ neighbours in each $V_{i\leq \tau}$.
There is a copy of  $K_t^{\epsilon_1 n_{t-1}/4}$ in $G$, implying $F\subset G$, a contradiction.
Thus we have $|Y_i|\leq  c_{\epsilon,F}/t$, implying $|Y|\leq c_{\epsilon,F}$.
Note that the degree of vertex of $V_k\setminus Y$ in every $P_j$ is less than $\sum_{i\neq j} |P_i\setminus V_k|+\sqrt{\epsilon_1}n_{t-1}\leq \sum_{i\neq i_k} |P_i\setminus V_k|+2\sqrt{\epsilon_1}n_{t-1}$.
Thus by \eqref{minimum degree in F free extremal}, the difference of the degrees of vertices of $G-Y$ in same class of $\mathcal{V}$ is at most $\epsilon n_{t-1}$.

Let $\delta_3=\sqrt[5]{\epsilon_1}\gg \delta_2$.
In the proof we show every vertex of $Y$ is adjacent to a  copy of $K_t^{\delta_3 n_{t-1}}$.
For vertex $v\in V(G)\setminus Y$, due to (\ref{minimum degree in F free extremal}) and the neighbours size of $a$ in its own part is less than $\epsilon n_{t-1}$, we have $|N_{K[\mathcal{P}]}(v)\bigtriangleup N_G(v)|\leq \epsilon n_{t-1}$.
We finish the proof the strong stability theorem.
\hfill$\square$ \medskip

\section{Applications}\label{prove conjecture}\label{section 6}
\noindent{\bf Proofs of Theorems~\ref{strong bollobas} and~\ref{conjecture}.}
Given $r \geq t \geq 3$ and $k \geq 2$, let $n_1 \geq \ldots \geq n_r$  and $n_{t-1}$ be sufficiently large.
Let $G \subseteq K_{n_1,\ldots,n_r}$ be an extremal graph for $kK_t$.
We will show that $e(G)=f(n_1,\ldots,n_r,k,t).$
Moreover,  the extremal graph is obtained from a complete $(t-1)$-partite graph in  $K_{n_1,\ldots,n_r}$ by joining all possible edges incident with $k-1$ fixed vertices.

Let $\epsilon>0$ be a small constant.
Since $e(G)\geq f(n_1,\ldots,n_r,1,t)$, by Theorem \ref{strong stability}  there exist   constants $\delta_1 \ll \delta_2 \ll \delta_3 \ll \epsilon$ and $c_{\epsilon,F}$ depending on and $r$, $\epsilon$ and $F$ such that if $n_{t-1}\geq 1 / \delta_1$ then
\begin{itemize}
\item [$(a)$] there exists an $(X,\delta_2)$-stable $(t-1)$-partition $\mathcal{P}$,
\item  [$(b)$] there exists a set $Y$ with $|Y|\leq c_{\epsilon,F}$ such that every vertex $v$ of $G-Y$ satisfies $|N_G(v)\bigtriangleup N_{K_{n_1,\ldots,n_r}[\mathcal{P}]}(v)|\leq \epsilon n_{t-1}$ and
 \item [$(c)$]  the difference of the degrees of vertices of $G-Y$ in same class of $\mathcal{V}$ is at most $\epsilon n_{t-1}$ and
\item  [$(d)$] every vertex of $Y$ is adjacent to at lest  $\delta_3 n_{t-1}$   to each class of $\mathcal{P}$ such that they induced a complete $(t-1)$-partite graph.
\end{itemize}

For an edge $xy$ inside $P_s \setminus Y$, we say $xy$ \textcolor{blue}{{\it good}} if $xy$ is adjacent to a copy of $K_{t-2}^{kt}$ consisting of edges between different $P_i$'s and say $xy$  \textcolor{blue}{{\it bad}} otherwise.
Let $E$ be the set of good edges.

\begin{center}
\begin{tikzpicture}[scale = 1]
\tikzstyle{every node}=[scale=0.7]

\draw  (0,0) ellipse  (0.4 and 1.2);
\draw  node at (-1,0) {$\widetilde{P}_1$};

\draw [purple] (0,2) circle (2pt);
\draw [purple] node at (-0.25,2.1) {$x$};

\draw [blue](0,1.5) circle (2pt);
\draw [blue] node at (-0.25,1.4) {$y$};

\draw [purple] (3,1.4) ellipse  (0.4 and 1.2);
\draw [purple] node at (3.8,1.4) {$\widetilde{V}_{i_1}$};

\draw [blue] (3,-1.4) ellipse  (0.4 and 1.2);
\draw [blue] node at (3.8,-1.4) {$\widetilde{V}_{i_2}$};

\draw [line width=0.1cm,   red ,opacity=1] (0,2)--(0,1.5);
\draw [line width=0.1cm,  dotted,black ,opacity=0.5] (0,1.5)--(2.6,1.4);
\draw [line width=0.1cm, dotted,black ,opacity=0.5] (0,2)--(2.6,-1.4);
\draw [line width=0.1cm, dotted,black ,opacity=0.5] (0.4,0)--(2.6,-1.4);
\draw [line width=0.1cm, dotted,black ,opacity=0.5] (0.4,0)--(2.6,1.4);

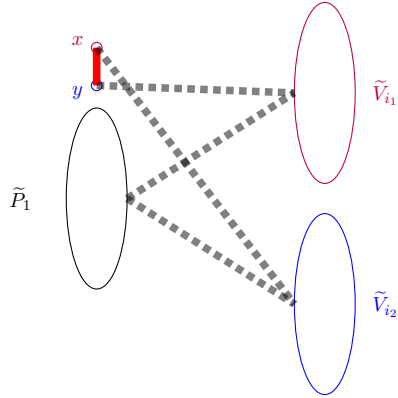
\draw node at (1.5,-3) {Figure 3. A  bad edge $xy$ with $x\in V_{i_1},y\in V_{i_2}$};
\end{tikzpicture}
\end{center}

\medskip

\noindent{\bf Claim 1.}   $|E|\leq k\epsilon n_{t-1}$.

\medskip

\begin{proof}
Clearly, there is no matching on $k$ edges in $E$, as otherwise we can easily find a copy of $kK_t$.
Since each vertex is adjacent to at most $ \epsilon n_{t-1}$ vertices in their own class of $\mathcal{P}$, we have $|E|\leq k\epsilon n_{t-1}$.
\end{proof}

\noindent{\bf Claim 2.} If $Y=\emptyset$, then $e(G)-|E| \leq f(n_1,\ldots,n_r,1,t)$. Moreover, the equality holds when $G-E$ is a $(t-1)$-partite graph.

\medskip

\begin{proof}
Let $\widetilde{V}_i=V_i \setminus X $, $\widetilde{P}_j= P_j \setminus X$, $X_{i,j}=\widetilde{V}_i\cap \widetilde{P}_j$ and $X_i=X \cap P_i$.
Let $m$ be the minimum size of the integral part of all classes of $\mathcal{P}_X$.

Let $xy$ be a bad edge in $P_s$ with $x \in V_{i_1}$ and $y\in V_{i_2}$.
Suppose that the size of $X_{i_1,j}$ is less than $m- 6\epsilon  n_{t-1}$ for each  $j\in[t-1]$ distinct from $s$.
By $(b)$, $x$ is adjacent to  $|\widetilde{P}_{j}\setminus \widetilde{V}_{i_1}|- \epsilon  n_{t-1}$ vertices of $\widetilde{P}_{j\neq s}$.
In particular, $x$ is adjacent to  at least $5\epsilon  n_{t-1}$ vertices of each integral part of $\widetilde{P}_{j\neq s}$.
If  $V_{i_2}\subseteq X$, $\widetilde{V}_{i_2}$ is partial in $\mathcal{P}_X$ or $\widetilde{V}_{i_2}$ is integral in $\widetilde{P}_s$, then $y$ is adjacent to all but at most $\epsilon  n_{t-1}$ vertices of $\widetilde{P}_{j\neq s}$ and hence $xy$ is adjacent to a copy of $K_{t-2}^{kt}$ consisting of edges between different $P_{j\neq s}$'s, a contradiction.

Now we may assume that $\widetilde{V}_{i_2}$ is integral in $\widetilde{P}_{q}$ with $q\neq s$.
Note that $|\widetilde{P}_{s}|\geq m$ and    $|\widetilde{P}_{q} \setminus \widetilde{V}_{i_2}| \leq m$ (by (a), $\mathcal{P}_X$ is stable to $\mathcal{V}_X$).
Consider vertices in $V_{i_2}$, it follows from (b) and (c) that $  |P_q|-2 \epsilon  n_{t-1}  \leq |\widetilde{P}_s| +|\widetilde{V}_{i_2}| \leq|P_q|+2 \epsilon  n_{t-1}$, implying $m \leq |\widetilde{P}_s |\leq m+3\epsilon  n_{t-1}$ and $m-3\epsilon  n_{t-1} \leq |\widetilde{P}_{q} \setminus \widetilde{V}_{i_2}| \leq m$.
Recall that $|X_{i_1,q}|\leq m- 6\epsilon  n_{t-1}$.
We can easily see that $xy$ is adjacent to a copy of $K_{t-2}^{kt}$ consisting of edges between different $P_{j\neq s}$'s, a contradiction.
Therefore, there exists a constant $\ell$ distinct from $s$  such that  $|X_{i_1,\ell}| \geq m- 6\epsilon  n_{t-1}$.
Remind that $m\leq |\widetilde{P}_s|\leq m+3\epsilon n_{t-1}$, thus the partial part size of $\widetilde{P}_s$ is less than $3\epsilon n_{t-1}$, implying $|X_{i_1,s}|\leq 3\epsilon n_{t-1}$.

Note $|\widetilde{P}_\ell\setminus (\widetilde{V}_{i_1}\cup \widetilde{V}_{i_2})|\leq \epsilon n_{t-1}$.
Consider vertices in $V_{i_2}$, applying (c) again, we have $|X_{i_1,\ell}|=|\widetilde{P}_s\setminus \widetilde{V}_{i_2}|\leq |\widetilde{P}_s|+\Theta( \epsilon  n_{t-1})=m+\Theta( \epsilon  n_{t-1})$.
Similarly, consider vertices in $V_{i_1}$, we have $|X_{i_2,\ell}|=|\widetilde{P}_s\setminus \widetilde{V}_{i_1}|=m+\Theta( \epsilon  n_{t-1})$.
Therefore, there exists an integer $\ell(i_1,i_2)=\ell$ such that
$$|X_{i_1,\ell(i_1,i_2)}|=|X_{i_2,\ell(i_1,i_2)}|=m+\Theta( \epsilon  n_{t-1}).$$
Otherwise, $xy$ is adjacent to a copy of $K_{t-2}^{kt}$ consisting of edges between different $P_{j\neq s}$'s, a contradiction.

Now we show that $P_{\ell(i_1,i_2)}\setminus (V_{i_1}\cup V_{i_2})=\emptyset$.
Suppose that exists a vertex $z\in P_{\ell(i_1,i_2)}\setminus (V_{i_1}\cup V_{i_2})$ and let $z\in V_{i_3}$.
Thus by $(b)$ we have $d_G(v)=\sum_{i\neq \ell}|P_i\setminus V_{i_3}|+\Theta(\epsilon n_{t-1})$.
If $V_{i_3}\subset X$ then $d_G(v)\geq \sum_{i\neq s}|P_i|+\Theta(\epsilon n_{t-1})$ implying $m\geq 2m+\Theta(\epsilon n_{t-1})$, a contradiction.
Let $\widetilde{V}_{i_3}$ has vertices in $\widetilde{P}_k$.
By $(b),(c)$ we have $|\widetilde{P}_k\setminus \widetilde{V}_{i_3}|=2m+\Theta(\epsilon n_{t-1})$.
By $(a)$ we have $|\widetilde{P}_k\setminus \widetilde{V}_{i_3}|\leq m$, a contradiction.
Thus $P_{\ell(i_1,i_2)}\setminus (V_{i_1}\cup V_{i_2})=\emptyset$.
We call such $P_{\ell(i_1,i_2)}$ a  \textcolor{blue}{{\it bad  class}} of the bad edge $xy$ and call such $(i_1,i_2)$ \textcolor{blue}{{\it bad  pair}}.

We can conclude there is no $i$ which appears in two bad pairs.
Otherwise, suppose we have bad pairs $(i_1,i_2),(i_2,i_3)$.
Due to $(a)$ we have $|X_{i_1,\ell(i_1,i_2)}|\geq |X_{i_2,\ell(i_1,i_2)}|+|X_{i_2,\ell(i_1,i_2)}|=2m+\Theta(\epsilon n_{t-1})$, a contradiction.

Let $G^\prime=G-E$.
Suppose that there is a bad edge $v_iv_j$ in $P_1$, otherwise we are done.
Let $v_i\in V_i$ and $v_j\in V_j$.
Clearly, $V_i \cap P_1$ and $ V_j\cap P_1$ forms a bipartite graph in $P_1$ and other vertex in $P_1$ is not adjacent to $(V_i \cup V_j) \cap P_1$.
Let $G^\ast$ be the graph obtained from $G^\prime$ by adding all possible edges between $P_i$'s and all edges between $V_i$ and $V_j$ (in the same class of $\mathcal{P}$) incident with a bad edge.

We can conclude $G^\ast$ is $K_t$-free.
Otherwise we have $K_t\subset G^\ast$.
Let $V(K_t)=\{a_1,\ldots,a_t\}$ with $a_i\in V_i$ for $i\in [t]$.
By Pigeonhole Principle, we can suppose $a_1,a_2\in P_1$ thus $a_1a_2$ form a bad edge.
Note $1$ cannot appear in other bad pair except $(1,2)$ thus $a_{i\geq 3}\notin P_1$.
Obviously, we have $a_{i\geq 3} \notin P_{\ell(1,2)}$.
Thus use the Pigeonhole Principle recursively and obtain vertex $a_t$ not belongs to any of $P_j$, a contradiction.
Thus $G^\ast$ is $K_t$-free.

If $|P_1 \setminus (V_i \cup V_j)|> |V_j\cap P_{\ell(i,j)}|$, then  similarly, we can show that $H_1=G^\ast_{V_i\cap P_1   \rightarrow S}$ with $S=\bigcup_{\iota\neq \ell(i,j)}P_i\setminus V_i $ keeps $K_t$-free with more edges than $G^\ast$.
Hence, by Theorem~\ref{extremal number 1}, $e(G^\prime) \leq e(G^\ast)  <  e(H_1)\leq f(n_1,\ldots,n_r,1,t)$.
Thus we may assume that  $|P_1 \setminus (V_i \cup V_j)|  \leq  \min \{|V_i\cap P_{\ell(i,j)}|,|V_j\cap P_{\ell(i,j)}|\}$.
Without loss of generality, let $|     V_i\cap (P_{\ell(i,j)}\cup P_1) | \leq |V_j\cap (P_{\ell(i,j)}\cup P_1)|$.
Now, let $P^\ast_1= (P_1 \setminus V_j)  \cup (V_i\cap P_{\ell(i,j)})  $ and $P^\ast_{\ell(i,j)}=V_j\cap (P_{\ell(i,j)}\cup P_1)$.
The graph $H_2 $ obtained from $G^\ast$ by changing $P_1$ and $P_{\ell(i,j)}$ to $P^\ast_1$ and $P^\ast_{\ell(i,j)}$ and replacing all edges between $P_1$ and $P_{\ell(i,j)}$ to all edges between $P^\ast_1$ and $P^\ast_{\ell(i,j)}$ (keeping other edges incident with $P_1$ and $P_{\ell(i,j)}$).
Again, we can see that $H_2$ keeps $K_t$-free with more edges than $G^\ast$.
by Theorem~\ref{extremal number 1}, $e(G^\prime) \leq e(G^\ast)  <  e(H_1)\leq g(n_1,\ldots,n_r,1,t)$.
We finish the proof of Claim 2.
\end{proof}

First, we prove the $k=1$ case, i.e., we first prove Theorem~\ref{strong bollobas}.
Clearly, $(d)$ implies that  $Y$ is empty.
Moreover, we have $|E|=0$.
If there is an edge in $P_i$, then by Claim 2 we get a contradiction.
Thus $G[P_i]$ is an independent set, and hence the result follows by Theorem~\ref{extremal number 1}.

We prove Theorem~\ref{conjecture} by induction on $k$.
Suppose the theorem holds for $k-1$.
Suppose that there is a vertex $a \in Y$.
By $(d)$, $a$ is adjacent to a copy of $K_{t-1}^{kt}$, thus $G-a$ must be $(k-1)K_t$-free.
Let $a  \in  V_{\iota}$.
By induction hypothesis, we have $$e(G)\leq e(G-a)+n-n_\iota\leq g(n_1,\ldots,n_\iota-1,\ldots,n_r,k-1,t)+n-n_\iota\leq g(n_1,\ldots,n_r,k,t),$$
the inequality holds only if $d_G(a)=n-n_\iota$ and the last inequality holds by the construction of joining a vertex of $V_\iota$ to the possible vertices in extremal graph reaches $g(n_1,\ldots,n_\iota-1,\ldots,n_r,k-1,t)$.
Thus by induction hypothesis, the extremal graph is a complete $(t-1)$-partite graph with $k$ vertices adjacent to all other vertices.

Therefore, we may assume that $Y$ is empty.
Combining with Claims 1 and 2, $e(G)\leq g(n_1,\ldots,1,t)+k\epsilon n_{t-1}<g(n_1,\ldots,n_r,k,t)$, a contradiction, hence we finish the proof of Theorem~\ref{conjecture}.
\hfill$\square$ 

\end{document}